\documentclass[12 pt]{article}
\usepackage{a4wide}
\usepackage{amsmath}
\usepackage{amsmath,amsthm,amsfonts}
\usepackage{pstricks, pst-plot,pst-node}
\usepackage{graphicx}
\usepackage{psfrag}

\newtheorem{thm}{Theorem}[section]
\newtheorem{theorem}[thm]{Theorem}
\newtheorem{cor}[thm]{Corollary}
\newtheorem{lem}[thm]{Lemma}
\newtheorem{lemma}[thm]{Lemma}
\newtheorem{prop}[thm]{Proposition}

\newtheorem*{conj}{Conjecture}

\newtheorem{corstar}[thm]{Corollary*}
\newtheorem{propstar}[thm]{Proposition*}
\newcommand{\Mbar}{\overline{M}}
\newcommand{\An}{\mathcal{A}_{n}}
\newcommand{\QQ}{\mathbb{Q}}

\begin{document}
\title{Gromov-Witten theory of $\An$-resolutions}
\author{Davesh Maulik}
\maketitle
\begin{abstract}
We give a complete solution for the reduced Gromov-Witten theory
of resolved surface singularities of type $A_{n}$, for any genus, with arbitrary descendent
 insertions.   We also present a partial evaluation of the $T$-equivariant relative Gromov-Witten theory of the threefold $\An\times\mathbf{P}^{1}$ which, under a nondegeneracy hypothesis,
yields a complete solution for the theory.  The results given here allow comparison of this theory with the quantum cohomology of the Hilbert scheme of points on the $\mathcal{A}_{n}$ surfaces.
We discuss generalizations to linear Hodge insertions and to surface resolutions of type $D,E$.
As a corollary, we present a new derivation of the stationary Gromov-Witten
theory of $\mathbf{P}^{1}$.
\end{abstract}
\tableofcontents
\section{Introduction}
\subsection{Overview}

Let $\zeta$ be a primitive $(n+1)$-th root of unity and consider the action of $\mathbb{Z}_{n+1}$
on $\mathbb{C}^2$ for which the generator acts via
$$(z_1,z_2) = (\zeta z_1, \zeta^{-1} z_2).$$  Let $\An$ be the minimal resolution
$$\An \rightarrow \mathbb{C}^2/\mathbb{Z}_{n+1}.$$  
The algebraic torus $T = (\mathbb{C}^{\ast})^2$
acts on $\mathbb{C}^2$ via the standard diagonal action. This commutes with the action of the cyclic group, 
so there is an induced $T$-action on the quotient singularity and its resolution $\An$.

The Gromov-Witten theory of $\An$ is defined by integrating cohomology classes against the virtual fundamental class
of the moduli space of stable maps $$\Mbar_{g}(\An, \beta).$$  Since $\An$ admits a holomorphic symplectic
form, it is a well-known fact that the virtual fundamental class vanishes and the Gromov-Witten theory is trivial. 
In the case of compact $K3$ surfaces, for example, this vanishing is a consequence of the existence of nonalgebraic deformations 
of the surface which do not contain any holomorphic curves.  In this case, one can correct for these deformations
by instead working with a reduced virtual class of dimension one larger than the usual expected dimension; for K3 and 
abelian surfaces, reduced Gromov-Witten invariants have been used to study enumerative conjectures of Yau and Zaslow.
\cite{bryan-leung}

In this paper, we completely solve the reduced Gromov-Witten theory of the $\An$ surfaces in all genus with arbitrary descendents.  In the case of $\mathcal{A}_{1}$, this solution was conjectured by R. Pandharipande
with motivation from the crepant resolution conjecture \cite{chen-ruan, bryan-graber}.  These
 surfaces have the striking property that the solution can be expressed by a closed formula.  This is in contrast to other varieties, such as a point or $\mathbf{P}^1$,
 for which a complete solution only exists via complicated recursions or differential equations.  In this sense, the $\An$ surfaces have the simplest known nontrivial Gromov-Witten theory. 

We also study the $T$-equivariant relative Gromov-Witten theory of the threefold $\An \times \mathbf{P}^1$.  We give a partial evaluation of relative invariants, corresponding to what we call divisor operators; under the assumption of a nondegeneracy conjecture in section \ref{generationconjecture}, this gives a solution for the complete relative theory.  
As a corollary, these divisor evaluations lead to closed formulas
for linear Hodge integrals in the reduced theory of the surface in terms of hypergeometric series.  
Our argument also yields a new derivation of the stationary theory of $\mathbf{P}^1$,
first studied in \cite{okpanp1}.

\subsection{Gromov-Witten theory of $\An$}

Viewed as a crepant resolution of a quotient singularity, the exceptional locus 
of $\An$ consists of a chain of $n$ rational curves $E_1, \dots, E_n$ with intersection matrix given by the Cartan matrix for the $A_n$ root lattice.  That is, 
each $E_{i}$ has self-intersection $-2$ and intersects $E_{i-1}$ and $E_{i+1}$ transversely.  These classes span $H^2(\An, \mathbb{Q})$ and, along with the identity class, span the full cohomology ring of $\An$.  For $1\leq i < j \leq n$, we define the effective curve classes
$$\alpha_{ij} = E_{i} + E_{i+1} + \dots E_{j-1}$$
corresponding to roots of the $\An$ lattice.

Given cohomology classes $\gamma_1, \dots, \gamma_m \in H^{\ast}(\An, \mathbb{Q})$, we are interested in descendent invariants in
the reduced Gromov-Witten theory
\begin{align*}
\langle \prod_{k=1}^{m} \tau_{a_{k}}(&\gamma_{k})\rangle^{\An,\mathrm{red}}_{g,\beta}
= \int_{[\Mbar_{g,m}(\An,\beta)]^{\mathrm{red}}} \prod_{k=1}^{m} \psi_{k}^{a_{k}}\mathrm{ev}^{\ast}(\gamma_k)
\end{align*}
where
$\psi_k \in H^{2}(\Mbar_{g,m}(\An,\beta),\mathbb{Q})$ is the first Chern class of the cotangent line bundle $L_k$ on the moduli space of maps associated to the $k$-th marked point.
Although $\An$ is noncompact, the moduli space
$\Mbar_{g,m}(\An, \beta)$ is compact for nonzero $\beta \in H_{2}(\An,\mathbb{Z})$.  The notation $[\Mbar_{g,m}(\An,\beta)]^{\mathrm{red}}$ 
refers to the \textit{reduced} virtual fundamental class on the moduli space, which has dimension
$$g+m.$$ 

Fix a curve class $\alpha= \alpha_{ij}$ and consider integers $a_1, \dots, a_r > 0$,
 $b_1, \dots, b_s \geq 0$, and divisor classes $\omega_1, \dots \omega_s \in H^2(\An, \mathbb{Q})$
 which satisfy the dimension constraint
$$\sum a_i + \sum b_j = g+r.$$ 
We prove the following evaluation:
\begin{theorem}\label{mainthm}
 For curve classes of the form $\beta = d \alpha$, we  have
 \begin{align}
 \langle \prod_{k=1}^{r} \tau_{a_{k}}(1) &\prod_{l=1}^{s} \tau_{b_{l}}(\omega_{l}) 
 \rangle^{\An, \mathrm{red}}_{g, d \alpha} =\nonumber\\ 
& \frac{(2g+r+s-3)!}{(2g+s-3)!}
 d^{2g+s-3} 
 \prod_{k=1}^r \frac{(a_k-1)!}{(2a_k-1)!}\left(-\frac{1}{2}
\right)^{a_k-1}\label{mainform}\\
&\cdot\prod_{l=1}^s \frac{b_l!}{(2b_l+1)!}\left(-\frac{1}{2}
\right)^{b_l} (\alpha \cdot \omega_l).\nonumber 
\end{align}
If $\beta$ is not a multiple of $\alpha$ for any root $\alpha$, then all reduced invariants vanish.
\end{theorem}

There are several nice qualitative features of this formula.  First, the answer is essentially independent
of which $\An$ surface we consider.  Second, while a priori the number of possible curve classes grows with $n$, 
we only need to look at multiples of roots and in fact 
the answer is essentially independent of the choice of root. 
Moreover, the degree dependence is monomial and the contributions of each insertion nearly factor completely.  Our strategy will be to prove these statements first and reduce the precise evaluation to the simplest possible case.  Using an argument due to Jim Bryan, the formula above can be extended
to resolutions associated to type $D$ and $E$ root lattices.

\subsection{Gromov-Witten theory of $\An\times\mathbf{P}^1$}

Consider the projective line $\mathbf{P}^1$ with $k$ distinct marked points $z_1, \dots, z_k$.  
Given a curve class $\beta \in H_2(\An, \mathbb{Q})$, an integer $m\geq 0$, and $k$ partitions
$$\mu_1,\dots, \mu_k$$
of $m$, the moduli space
$$\Mbar^{\bullet}_{g}(\An \times \mathbf{P}^1, (\beta,m); \mu_1, \dots, \mu_k)$$
consists of (possibly disconnected) relative stable maps from genus $g$ curves to the threefold $\An \times \mathbf{P}^1$, with target homology class given by $(\beta, m) \in H_{2}(\An\times\mathbf{P}^1,\mathbb{Z})$ and
 with ramification profile given by the
partition $\mu_i$ over each divisor $\An\times z_i$.   We assume that the ramification points over each relative divisor are marked and ordered; the $\bullet$ here follows the notation from \cite{localcurves} and indicates that we do not allow collapsed connected components in the domain.  Unlike the previous section, where we only considered reduced theory of the surface, we are now interested in the full $T$-equivariant theory of the threefold.  This space possesses a virtual fundamental class of dimension
$$ -K_{\An}\cdot\beta + 2m +\sum_i(l(\mu_i) - m) = \sum_i l(\mu_i) + (2-k)m.$$

Given a nonnegative integer $m$, a \textit{cohomology-weighted} partition of $m$ consists of an unordered set of pairs
$$\overrightarrow{\mu} = \{(\mu^{(1)}, \gamma_1), \dots, (\mu^{(l)}, \gamma_l)\}$$
where $\{\mu^{(1)},\dots,\mu^{(l)}\}$ is a partition whose parts are labelled by cohomology classes $\gamma_i \in H^{\ast}(\An, \mathbb{Q})$.

Suppose that we have $k$ weighted partitions of $m$:
$$\overrightarrow{\mu_{1}},\dots,\overrightarrow{\mu_{k}}.$$
For each part $\mu_{r}^{(s)}$ of the partition $\mu_r$, there is an associated cohomology class
$\gamma_{r}^{(s)}$ on $\An$ as well as an evaluation map
$$\Mbar^{\bullet}_{g}(\An \times \mathbf{P}^1, (\beta,m);\mu_1,\dots \mu_k) \longrightarrow \An \times z_r = \An$$
associated to the corresponding ramification point.
We define relative invariants by pulling back each cohomology class by its associated evaluation map:
\begin{align*}
\langle \overrightarrow{\mu_1}, \dots,&\overrightarrow{\mu_k}\rangle^{\An\times\mathbf{P}^1}_{g,\beta}
= \frac{1}{\prod |\mathrm{Aut}(\mu_r)|}
\int_{[\Mbar_g(\An\times\mathbf{P}^1)]^{\mathrm{vir}}}\prod_{r=1}^{k}\prod_{s=1}^{l(\mu_r)} \mathrm{ev}^{\ast}\gamma_{r}^{(s)}.
\end{align*}
The automorphism prefactor corrects for the fact that our relative conditions are \textit{unordered} partitions while ordered partitions are required to define the moduli space and evaluation maps.
In the case where $\beta = 0$, the space of relative stable maps  is not compact in which case this integral must be defined as a localization residue with respect to the $T$-action, as explained in
\cite{localcurves}.  

We can encode these relative invariants in a partition function
\begin{align*}
\mathsf{Z}^{\prime}(\An\times\mathbf{P}^1)_{\overrightarrow{\mu_1},\dots, \overrightarrow{\mu_k}} &= \sum_{g,\beta} 
\langle \overrightarrow{\mu_1},\dots,\overrightarrow{\mu_k}\rangle^{\An\times\mathbf{P}^{1}}_{g,\beta} u^{2g-2} s^{\beta}
\in \mathbb{Q}(t_1,t_2)((u))[[s_1,\dots,s_n]]
\end{align*}
where $s^{\beta} = \prod_{i=1}^{n} s_i^{\beta\cdot \omega_{i}}$
and $\{\omega_{1},\dots,\omega_{n}\}$ is the dual basis to
$\{E_{i}\}$ in $H^{2}(\An,\mathbb{Q})$ under the Poincare pairing.
  Again, the notation here for 
follows that of \cite{localcurves}.

Using the results of the previous section, we calculate
$$\mathsf{Z}^{\prime}(\An\times\mathbf{P}^{1})_{\overrightarrow{\mu},\overrightarrow{\rho},\overrightarrow{\nu}}$$
for 
$$\overrightarrow{\rho} = \{(1,1)^{m}\}, \{(2,1)(1,1)^{m-2}\}\quad\mathrm{or}\quad\{(1,\omega_{i})(1,1)^{m-1}\}.$$
These relative conditions correspond to unit and divisor operators for the Hilbert scheme of points on $\An$, under the Gromov-Witten/Hilbert correspondence discussed in section \ref{relan}.
Assuming a nondegeneracy conjecture for these operators, we explain how to determine
the full partition function above in terms of these evaluations and gluing relations from
the degeneration formula, in a manner analogous to the local curve theory of Bryan
and Pandharipande \cite{localcurves}.  One can extend these results further to twisted $\An$-bundles
over a genus $g$ curve.  This theory can be viewed as a deformation of the enriched TQFT structure
described in that paper.

\subsection{Relation to other theories}

As we explain in section \ref{ringstructure}, using a construction that is valid for any surface, the relative Gromov-Witten theory of $\An\times\mathbf{P}^{1}$ 
induces a ring structure on
$$H_{T}^{\ast}(\mathrm{Hilb}_{m}(\An),\QQ)\otimes \QQ(t_1,t_2)((u))[[s_{1},\dots,s_{n}]]$$
that is a deformation of the classical cohomology ring of the Hilbert scheme of points on $\An$.
 Our work here is the starting point of a series of comparisons of this ring to related theories.  In related work with A. Oblomkov(\cite{hilban},\cite{dtan}), we will prove a triangle of equivalences between the Gromov-Witten theory of $\An\times \mathbf{P}^1$,
the Donaldson-Thomas theory of $\An\times \mathbf{P}^1$, and the quantum
cohomology of the Hilbert scheme of points on the $\An$ surface, each of which
provides a ring deformation of the classical cohomology of the Hilbert scheme. We will
explain the Gromov-Witten/Hilbert correspondence for $\An$ in detail in section \ref{relan}.

\begin{figure}[hbtp]\psset{unit=0.5 cm}                                                                         
  \begin{center}                                                                                                
    \begin{pspicture}(-6,-2)(10,6)%\showgrid                                                                    
    \psline(0,0)(2,3.464)(4,0)(0,0)                                                                             
    \rput[rt](0,0){                                                                                             
        \begin{minipage}[t]{3.64 cm}                                                                            
          \begin{center}                                                                                        
            Gromov-Witten \\ theory of $\An \times \mathbf{P}^1$                                                  
          \end{center}                                                                                          
        \end{minipage}}                                                                                         
    \rput[lt](4,0){                                                                                             
        \begin{minipage}[t]{3.64 cm}                                                                            
          \begin{center}                                                                                        
             Donaldson-Thomas\\ theory of $\An \times \mathbf{P}^1$                                               
          \end{center}                                                                                          
        \end{minipage}}                                                                                         
    \rput[cb](2,4.7){                                                                                           
        \begin{minipage}[t]{4 cm}                                                                               
          \begin{center}                                                                                        
           Quantum cohomology \\ of $Hilb(\An)$                                                                 
          \end{center}                                                                                          
        \end{minipage}}                                                                                         
    \end{pspicture}                                                                                             
  \end{center}                                                                                                  
\end{figure}      

The above triangle was first shown to hold for $\mathbb{C}^2$ in \cite{localcurves},\cite{okpanhilb},\cite{okpandt}.  While the GW and DT vertices are always conjectured to be 
equivalent for arbitrary threefolds, the relationship with the quantum cohomology of the Hilbert scheme breaks down for general surfaces in the specific form we describe here.  Our work for $\An$ surfaces
provide the only other examples for which this triangle is known to hold.  

These equivalences play an essential role in proving the primary Gromov-Witten/Donaldson-Thomas 
correspondence for all toric varieties \cite{MOOP}.  The argument there 
provides an effective algorithm that computes primary Gromov-Witten invariants for arbitrary toric threefolds 
starting with the precise calculations of $\An\times\mathbf{P}^{1}$ of this paper as input.  In particular, the results here lead to expressions
for the two-leg and three-leg equivariant vertices.

In addition to these equivalences, 
the $\An$ higher genus evaluation we give here should be identical to
the higher genus orbifold theory of the Deligne-Mumford stack $[\mathbb{C}^2/\mathbb{Z}_{n+1}]$.
More precisely, our formulas when applied to the crepant resolution conjecture \cite{bryan-graber} yield
a conjectural evaluation of certain Hurwitz-Hodge integrals on the moduli space of $n+1$-fold covers
of genus $g$ curves.  We plan to investigate these evaluations in future work.

\subsection{Outline}
In section \ref{proofofmainthm}, after explaining preliminary features of the reduced theory, we prove the main theorem of the evaluation for $\An$.  In section \ref{rubberan}, we explain the evaluation of the $T$-equivariant theory of a nonrigid $\An\times \mathbf{P}^1$.  As a corollary of this argument, we present a new derivation of the stationary theory of $\mathbf{P}^1$ in terms of certain double Hurwitz numbers.  In section \ref{relan}, we
explain how to use these basic integrals to calculate the divisor operators discussed above and the generation conjecture that allows us to reconstruct the full relative theory of the threefold.  We also discuss the relationship with the quantum cohomology of $\mathrm{Hilb}(\An)$.  Finally, in section \ref{linearhodge}, we use these basic integrals to study linear Hodge series in the reduced theory of the surface, where we again obtain essentially closed expressions.

\subsection{Acknowledgements}

We wish to thank A. Oblomkov, A. Okounkov, and R. Pandharipande for many useful discussions and comments. We also thank J. Bryan for suggesting the proof in section \ref{dandesection}.
The author was partially supported by an NSF Graduate Fellowship and a Clay Research Fellowship.

\section{Proof of Theorem \ref{mainthm}}\label{proofofmainthm}

We prove theorem \ref{mainthm} in several steps.  We first study the reduced virtual class and explain
how it is equivalent to the linear part of the full $T$-equivariant theory of $\An$.  We then prove the degree scaling and root-independence properties of the evaluation by a localization argument.  This reduces the problem to the case of $\mathcal{A}_1$, where we finish the proof using exact calculations and a set of Virasoro relations for the reduced theory.  We close with an argument, suggested us
by Jim Bryan, reducing the case of surface resolutions of type $D$ and $E$ to the invariants calculated here.

\subsection{Notation}
Let us fix notation for our surfaces.  Recall that the exceptional locus of $\An$ is given
by a chain $E_1,\dots,E_n$  of rational $(-2)$-curves.  Under the $T$-action, there are $n+1$ fixed points $p_1, \dots, p_{n+1}$; the tangent weights 
at the fixed point $p_i$ are $(n+2-i)t_1 - (i-1)t_2$ and $(i-n-1)t_1 + it_2$.  
The $E_i$ are the $T$-fixed curves
joining $p_{i}$ to $p_{i+1}$.  We denote by $E_{0}$ and $E_{n+1}$ for the noncompact $T$-fixed
curve direction at $p_{1}$ and $p_{n+1}$ respectively.  On the $\mathcal{A}_{1}$ surface, we denote the exceptional curve by $E = E_1$ and its
Poincare dual by $\omega = \frac{-1}{2}[E]$.  
\begin{figure}
\begin{center}
\begin{picture}(300,130)(0,0) \put(150,30){\line(-3,1){60}}
\put(90,50){\line(-1,1){50}} \put(64,52){$3t_1$}
\put(90,50){\vector(-1,1){20}} \put(90,50){\vector(3,-1){25}}
\put(150,30){\vector(-3,1){25}} \put(72,36){$2t_2-t_1$}
\put(105,25){$t_1-2t_2$} \put(90,50){\circle{2}}
\put(90,50){\circle{1}} \put(90,55){$p_1$}
\put(150,30){\circle{1}}\put(150,30){\circle{2}}
\put(150,35){$p_2$}
\put(150,30){\line(3,1){60}}\put(150,30){\vector(3,1){25}}
\put(210,50){\vector(-3,-1){25}} \put(210,50){\line(1,1){50}}
\put(210,50){\vector(1,1){20}}\put(210,50){\circle{1}}
\put(210,50){\circle{2}} \put(202,55){$p_3$} \put(223,52){$3t_2$}
\put(163,25){$2t_1-t_2$}\put(198,36){$t_2-2t_1$}
\end{picture}
\caption{Tangent weights for $\mathcal{A}_{2}$}
\end{center}
\end{figure}
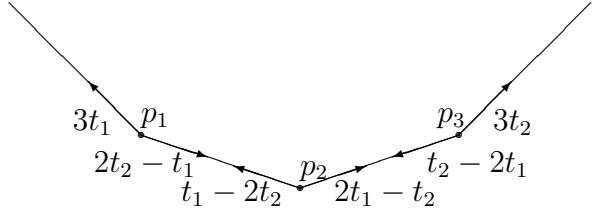
\subsection{Reduced classes}
In this section, we define the reduced virtual fundamental class for
$\Mbar_{g,m}(\An,\beta)$.  We also explain a comparison statement between the reduced 
Gromov-Witten theory and the linear part of the 
$T$-equivariant Gromov-Witten theory, as defined in the usual sense.  The 
algebraic construction given here is due originally to Behrend-Fantechi \cite{bfreduced}.  
Our discussion closely follows the more detailed treatement given in \cite{okpanhilb}.  

Heuristically, given any variety with a nonvanishing holomorphic symplectic form, this form gives rise to a trivial factor of the obstruction theory which leads to the vanishing of the usual nonequivariant virtual fundamental class.  By removing this trivial factor by hand, we obtain a nontrivial theory with virtual dimension increased by $1$.  In the context of compact K3 surfaces, a symplectic construction can also be given in terms of family Gromov-Witten invariants of the associated twistor fibration \cite{bryan-leung}.

We first explain the standard and modified obstruction theory for a fixed domain curve $C$.  Given a fixed nodal, pointed curve $C$ of genus $g$, let
$M_{C}(\An, \beta)$
denote the moduli space of maps from $C$ to $\An$ of degree $\beta \neq 0$.  The usual perfect obstruction theory for $M_C(\An,\beta)$ is defined by the natural morphism
\begin{equation}\label{standardmap}
R\pi_{\ast}(\mathrm{ev}^{\ast}T_{\An})^{\vee} \rightarrow L_{M_{C}},
\end{equation}
where $L_{M_{C}}$ denotes the cotangent complex of $M_{C}(\An,\beta)$ and
\begin{gather*}
\mathrm{ev}: C \times M_C(\An,\beta) \rightarrow \An,\\
\pi: C \times M_C(\An, \beta) \rightarrow M_C(\An,\beta).
\end{gather*}
are the evaluation and projection maps.

Let $\gamma$ denote the holomorphic symplectic form on $\An$ induced by the standard form
$dx\wedge dy$ on $\mathbb{C}^2$.  The $T$-representation $\mathbb{C}\cdot \gamma$
has weight $-(t_1+t_2)$. Let $\Omega_{\pi}$ and $\omega_{\pi}$ denote the sheaf of relative differentials and the relative dualizing sheaf.  The canonical map
$$\mathrm{ev}^{\ast}(\Omega_{\An}) \rightarrow \Omega_{\pi} \rightarrow \omega_{\pi}$$
and the symplectic pairing
$$T_{\An} \rightarrow \Omega_{\An} \otimes (\mathbb{C}\gamma)^{\vee}.$$
induce a map of bundles
$$\mathrm{ev}^{\ast}(T_{\An})\rightarrow \omega_{\pi}\otimes(\mathbb{C}\gamma)^{\vee},$$
This, in turn, yields a map
of complexes
$$R\pi_{\ast}(\omega_{\pi})^{\vee}\otimes\mathbb{C}\gamma \rightarrow R\pi_{\ast}(\mathrm{ev}^{\ast}(T_{\An})^{\vee})$$
and the truncation
$$\iota:\tau_{\leq -1}R\pi_{\ast}(\omega_{\pi})^{\vee}\otimes\mathbb{C}\gamma \rightarrow
R\pi_{\ast}(\mathrm{ev}^{\ast}(T_{\An})^{\vee}).$$
This truncation is a trivial line bundle with equivariant weight $-(t_1+t_2)$.

Results of Ran and Manetti (\cite{ran,manetti}) on obstruction theory and the semiregularity map imply the following.  First, there is an induced map
\begin{equation}\label{reducedmap}
C(\iota) \rightarrow L_{M_{C}}
\end{equation}
where $C(\iota)$ is the mapping cone associated to $\iota$.  Second, this map
(\ref{reducedmap}) 
satisfies the necessary properties of a perfect obstruction theory.  This is 
precisely the modified obstruction theory we use to define the
reduced virtual class.   Since all maps in this section are compatible with
the $T$-action, we have a $T$-equivariant reduced virtual class.

There is one important subtlety regarding the semiregularity results of (\cite{ran,manetti}).  
In order to apply their results, we require a compact target space.  We can embed
the $\An$ singularity in a surface with a holomorphic symplectic form that is
degenerate away from the singularity.  In the resolution, our curve maps entirely to the
nondegenerate locus, so theorem 9.1 of \cite{manetti} still gives the necessary vanishing statement for
realized obstructions.

As with the standard obstruction theory (\ref{standardmap}), we obtain
the \textit{reduced}  $T$-equivariant perfect obstruction theory on $\Mbar_{g,m}(\An,\beta)$ by varying the domain $C$,
and studying the the relative obstruction theory over the Artin stack $\mathfrak{M}$ of all nodal curves.  Since
the new obstruction theory differs from the standard one by 
the 1-dimensional obstruction space $(\mathbb{C}\gamma)^{\vee}$, we have
that the reduced virtual dimension is
$$1+ (g-1)+m.$$
Furthermore we have the identity
\begin{align*}
[\Mbar_{g,m}(\An,\beta)]^{vir}_{\mathrm{standard}} 
&= c_{1}(\mathbb{C}\gamma^{\vee})[\Mbar_{g,m}(\An,\beta)]^{\mathrm{red}}\\
&=(t_1+t_2)[\Mbar_{g,m}(\An,\beta)]^{\mathrm{red}}
\end{align*}

We have proven the following
\begin{lemma}\label{reducedlemma}
The standard $T$-equivariant Gromov-Witten invariants of $\An$ with nonzero degree are divisible by $(t_1+t_2)$.  Nonequivariant reduced Gromov-Witten invariants are encoded in the coefficient of $(t_1+t_2)$ in the full $T$-equivariant standard theory.
\end{lemma}

Finally, we close with a further comparison lemma in the case of $\mathcal{A}_{1}$.  In this case, the
surface is the cotangent bundle to $\mathbf{P}^1$.  For $d>0$, we
have an identification of moduli spaces
$$\Mbar_{g,m}(\mathcal{A}_{1}, d[E]) = \Mbar_{g,m}(\mathbf{P}^1,d).$$  
We can express the reduced virtual class of the left-hand side in terms of the virtual class of the
right-hand side with the following corollary.

\begin{cor}\label{reducedcor2}
$$[\Mbar_{g,m}(\mathcal{A}_{1},d[E])]^{\mathrm{red}} = 
c_{g+2d-2}(R\pi_{\ast}\mathrm{ev}^{\ast}\mathcal{O}(-2))[\Mbar_{g,m}(\mathbf{P}^1,d)].$$
\end{cor}
\begin{proof}
By lemma \ref{reducedlemma}, we want to calculate the linear part of the standard $T$-equivariant theory of the total space of $\mathcal{O}(-2)$.  The obstruction theory of this space differs from that of $\mathbf{P}^1$ by
the total Chern class $c(R\pi_{\ast}\mathrm{ev}^{\ast}\mathcal{O}(-2))$.  It suffices to check that the linear part of
this expression is precisely the penultimate Chern class of degree $g+2d-2$.     
\end{proof}

\subsection{Degree dependence}

We first analyze the degree dependence for the $\mathcal{A}_{1}$ surface and reduce the general $\An$ 
surface to this case.

\begin{prop}\label{degdep} 
$$\langle \prod_{i=1}^{r} \tau_{a_{i}}(1) \prod_{j=1}^{s} \tau_{b_{j}}(\omega) 
 \rangle^{\mathcal{A}_{1}, \mathrm{red}}_{g, d} =  d^{2g+s-3} \cdot
  \langle \prod_{i=1}^{r} \tau_{a_{i}}(1) \prod_{j=1}^{s} \tau_{b_{j}}(\omega) 
 \rangle^{\mathcal{A}_{1}, \mathrm{red}}_{g, 1}$$
\end{prop}
\begin{proof}
To simplify the analysis, assume $r=0$.  
By the results of the last section, we can compute
these invariants by virtual localization and extract the term proportional to $(t_1+t_2)$.  Note that we have the equality
$$\omega = E_0 +t_2 = E_2 + t_1,$$ where
$E_0, E_2$ are the noncompact $T$-fixed divisors at the fixed points $p_1, p_2$.  As 
already discussed, any invariant must be divisible by $(t_1+t_2)$.  Therefore, if we replace any of the divisors in our invariant by $1$, the invariant will vanish for dimension reasons.
In particular, we can replace $\omega$ with either $E_0$ or $E_2$ without affecting the answer.  Let us assume we have replaced them with $E_0$; in section \ref{stationarysection}, it will be useful to consider different combinations of these insertions.

Virtual localization expresses the invariant as a sum over a large number
of connected components of fixed loci.  Each such component consists of curves contracted over 
a fixed point of $\mathcal{A}_{1}$ along with edges corresponding to rational curves totally ramified over $E$.  The key observation is that only graphs with a single edge contribute to the linear term and each of these graphs has the same $d$-dependence.  We refer the reader to \cite{grabpan}
for a detailed explanation of the contributions to virtual localization.

\begin{figure*}[htp]
\centering
\psfrag{d}{$d$}
\psfrag{p}{$\mathbf{P}^1$}
\psfrag{m}{$\mu$}
\psfrag{n}{$\nu$}
\includegraphics[scale=0.50]{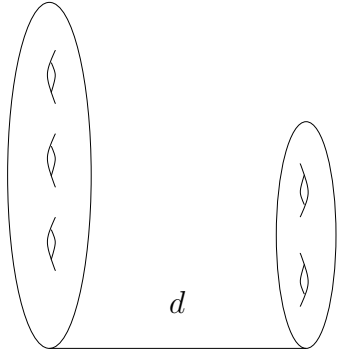}
\caption{Localization configuration}
\end{figure*}

Each edge corresponds to a totally ramified rational curve mapping to $E$ with degree $a > 0$.  The contribution of this edge to the localization term for this graph is the product of weights for $H^1(\mathbf{P}^1, \mathcal{O}(-2a))$:
$$2t_1, 2t_1+ \frac{t_2-t_1}{a},\dots (t_1+t_2),\dots, 2t_2 - \frac{t_2-t_1}{a}, 2t_2.$$
Therefore each edge contributes a factor of $(t_1+t_2)$.  Moreover, it is easy to see that all weights that occur in the denominator are of the form $it_1+jt_2$ where $i\cdot j \leq 0$.  Since
we are trying to calculate the linear term of an equivariant polynomial, it suffices to calculate the localization sum modulo $(t_1+t_2)^{2}$, in which case only graphs with a single edge contribute.

These graphs consist of a single contracted curve of genus
$g_1$ over $p_1$ which contains the $s$ marked points, a single contracted curve of genus $g_2= g-g_1$ over $p_2$, and a single edge of degree $d$ connecting them.  The contribution of vertex over $p_1$ of this graph is

$$(t_2-t_1)^{s}\int_{\Mbar_{g_1,s+1}} \prod_{i=1}^{s}\psi_{i}^{b_{i}}\frac{\Lambda^{\vee}(2t_1)\Lambda^{\vee}(t_2-t_1)}{\psi_{s+1} - \frac{t_2-t_1}{d}}.$$

In this expression, $\Lambda(t) = \lambda_g  + \lambda_{g-1} t + \dots  + t^{g}$ is the Chern
polynomial of the Hodge bundle on $\Mbar_{g,n}$.

Similarly, the contribution of the vertex over $p_2$ of this graph is
$$\int_{\Mbar_{g_2,1}} \frac{\Lambda^{\vee}(2t_2)\Lambda^{\vee}(t_1-t_2)}{\psi - \frac{t_1-t_2}{d}}.$$
The edge contribution is given by 
$$\frac{1}{d}\frac{(t_1+t_2)\prod_{k=1}^{d-1}((t_1+t_2)+ \frac{k(t_2-t_1)}{d})\cdot((t_1+t_2)- \frac{k(t_2-t_1)}{d})}{\prod_{k=1}^{d}\frac{k(t_2-t_1)}{d}\frac{-k(t_2-t_1)}{d}}.$$

We now analyze the $d$-dependence.
Again, since we are working modulo $(t_1+t_2)^2$ and the edge term carries a factor of $(t_1+t_2)$, 
we can calculate the vertex terms and the rest of the edge factors modulo $(t_1+t_2)$.  The $d$-dependence of the
edge term is $1/d$.  For the vertex terms, by Mumford's relation on Hodge classes, we have that
$$\Lambda^{\vee}(2t_1)\Lambda^{\vee}(t_2-t_1) \equiv (-1)^{g_i}(2t_1)^{2g_{i}} \mod (t_1+t_2).$$
Therefore the $d$-dependence is given entirely from the cotangent lines in the denominator of each expression.  For the vertex over $p_1$,
the exponent of $d$ is
$$3g_1 +s +1 - 3 - \sum b_i+1 = 3g_1 - g+s-1$$
and for the vertex over $p_2$ this is
$$3g_2 - 2+1.$$  The total $d$-dependence is exactly
$d^{2g+s-3}$.

If we include descendents of $1$ or Hodge classes, the argument applies unchanged.
\end{proof}

\subsection{$\An$ dependence}

The same localization argument allows us to reduce the $\An$ geometry to the $\mathcal{A}_{1}$ surface as described in Theorem \ref{mainthm}.  We give another argument for this reduction in section \ref{dandesection}.

\begin{prop}
If $\beta = d\alpha$ for a root class $\alpha = \alpha_{ij}$ then we have
$$\langle \prod_{k=1}^{r} \tau_{a_{k}}(1) \prod_{l=1}^{s} \tau_{b_{l}}(\omega_{l}) 
 \rangle^{\An, \mathrm{red}}_{g, \beta} =  d^{2g+s-3} \cdot \prod_{l=1}^{s}(\alpha\cdot \omega_{l})\cdot
  \langle \prod_{k=1}^{r} \tau_{a_{k}}(1) \prod_{l=1}^{s} \tau_{b_{l}}(\omega) 
 \rangle^{\mathcal{A}_{1}, \mathrm{red}}_{g, 1}.$$
 Otherwise, the reduced theory vanishes.
\end{prop}
\begin{proof}
Again, to simplify the analysis we ignore 
descendents of $1$.  As before, every edge in a given localization graph carries a factor
of $(t_1+t_2)$.  However, unlike the $\mathcal{A}_{1}$ analysis, there are vertex contributions that
have factors of $(t_1 + t_2)$ in the denominator.  When an edge of degree $a$ and an edge of degree $b$ meet at a fixed point with tangent weights $v_1$ and $v_2$ without a contracted curve joining them, there is a factor of 
$$\frac{v_1}{a} + \frac{v_2}{b}$$ 
to the fixed locus corresponding to smoothing that node.  In our case, this will be proportional
to $t_1+t_2$ if and only if the tangent directions are distinct and $a=b$.

Therefore, the multiplicity of $(t_1+t_2)$ in our graph is at least the number of edges minus the number of these special nodes. This is always positive and the graphs with multiplicity 1 have the following form.  There is a curve of genus $g_1$ contracted to a fixed point $p_{i}$, followed by
a chain of rational curves mapping with degree $d$ to $E_{i}, \dots, E_{j-1}$, and a curve of genus $g-g_1$ contracted to the fixed point $p_{j+1}$.  In particular, if $\beta$ is not a multiple of a root, there
are no such localization graphs and the reduced invariant vanishes.  If $\beta = d \alpha_{ij}$, then the $d$-dependence is again the same for every relevant graph.

Finally, for the divisor insertions, we can assume $\beta = d \alpha_{1,n+1}$.  Since all marked points must map to either $p_1$ or $p_{n+1}$, any divisor insertion $E_2, \dots, E_{n-1}$ gives vanishing.  This is consistent with the fact that $E_{k}\cdot \alpha_{1,n+1} = 0$ for these divisors.  A direct computation shows that the localization graph contribution with divisors $E_1$ and $E_n$ is identical to
the corresponding graph contribution on $\mathcal{A}_{1}$ with $-E_0$ and $-E_2$ insertions.  As discussed, these can be replaced with $\omega$ insertions; the signs are accounted for by
the fact that $E_1\cdot \alpha_{1,n+1} = E_n \cdot \alpha_{1,n+1} = -1$.

Again, adding descendents of $1$ and Hodge classes does not affect the argument.
\end{proof}

\subsection{Stationary descendents}
We have reduced the theorem to the case of $\mathcal{A}_{1}$, degree $d=1$.  It is convenient
to treat this case using the expression from corollary \ref{reducedcor2} for the reduced class in terms of $\Mbar_{g}(\mathbf{P}^1,d)$
with obstruction bundle induced by $\mathcal{O}(-2)$.
The divisor insertion $\omega$ on $\mathcal{A}_{1}$ is the cohomology class of a
point in $\mathbf{P}^1$.  We first assume there are no descendents of $1$, i.e. the stationary case.

\begin{prop}\label{statprop}
$$\langle \prod_{j=1}^{s} \tau_{b_{j}}(\omega) 
 \rangle^{\mathcal{A}_{1}, \mathrm{red}}_{g, 1} = \prod_{j=1}^s \frac{b_j!}{(2b_j+1)!}\left(-\frac{1}{2}
\right)^{b_j}.$$
\end{prop}
\begin{proof}
We evaluate using the degeneration formula.  There exists a degeneration of $\mathcal{A}_{1}$ to a comb configuration consisting of a central $\mathbf{P}^1$
with normal bundle $\mathcal{O}(-2)$ and $s$ rational teeth with trivial normal bundle.  In the
degenerate limit, each $\omega$ insertion lies on a distinct tooth.  Since our map has degree $1$, there
is no need to sum over possible configurations or relative conditions.  

The obstruction class insertion degenerates to
$c_{g+2d-2}(R\pi_{\ast}\mathrm{ev}^{\ast}\mathcal{O}(-2))$ on the spine and $\lambda_{h_i}$ on each tooth, where $h_i$ is the genus of the domain curve.  The curve mapping to the spine is forced to be genus $0$ and its contribution is clearly $1$.  As a result, the degree $1$ computation in the stationary case is multiplicative in its insertions.

Finally, in the case $s=1$, where $g= b_1$, we compute directly through a localization calculation sketched in the
proof of proposition \ref{degdep}.  We want the coefficient of $u^{3g}$
in the following product
$$\left( \sum_{g_1=0}^{\infty} \langle\tau_g \tau_{3g_1-g}\rangle u^{3g_1}\right)\cdot \left(\sum_{g_2=0}^{\infty} \langle \tau_{3g_2-2}\rangle u^{3g_2}\right)$$
where we are using bracket shorthand for integrals on $\Mbar_{g_i,n}$.  Both of these power series have already been computed in \cite{faber-pandharipande}.
It is easy to isolate the desired coefficient as
$$\frac{g!}{(2g+1)!}(-\frac{1}{2})^{g}.$$
\end{proof}

\subsection{Virasoro rule for general insertions}
It remains to prove theorem \ref{mainthm} for $\mathcal{A}_{1}$, degree 1, with descendents of $1$.  In non-equivariant Gromov-Witten theory, these insertions can conjecturally be removed using Virasoro constraints.  However, there is not even a conjectural picture of Virasoro constraints for equivariant Gromov-Witten theory or reduced Gromov-Witten theory.  Instead, our strategy is to use degeneration
arguments to embed our reduced invariants in 
toric projective Fano surfaces where these constraints exist and are well understood.

More precisely, we will prove the following Virasoro-type relation.  This relation will uniquely determine 
the full degree 1 theory of $\mathcal{A}_{1}$ in terms of the stationary case.  A direct calculation shows that 
equation (\ref{mainform}) is the unique solution.  Of course, equation (\ref{mainform})
gives a much simpler removal rule for descendents of $1$ but
we know of no direct proof. 

In what follows, we write 
$[\alpha]^{p}_{q}$ for the coefficient of $x^q$ in 
$(x+\alpha)(x+\alpha+1)\dots(x+\alpha+p)$.

\begin{prop}
\begin{align*}
\langle \tau_{a+1}(1) \prod_{i=1}^{r} &\tau_{a_{i}}(1) \prod_{j=1}^{s} \tau_{b_{j}}(\omega)
\rangle^{\mathcal{A}_{1},\mathrm{red}}_{g,1} \cdot \left[\frac{1}{2}\right]^{a}_{0}
= \\
&-(2a+2)
\langle \tau_{a}(\omega) \prod_{i=1}^{r} \tau_{a_{i}}(1)\prod_{j=1}^{s} \tau_{b_{j}}(\omega)\rangle^{\mathcal{A}_{1},\mathrm{red}}_{g,1}\\
&+\sum_{i=1}^{r}  \left[a_{i}-\frac{1}{2}\right]^{a}_{0}
\langle \tau_{a_{i}+a}(1) \prod_{k \neq i} \tau_{a_{k}}(1)\prod_{j=1}^{s} \tau_{b_{j}}(\omega)\rangle^{\mathcal{A}_{1},\mathrm{red}}_{g,1}\\
&+\left(\left[a_{i}+\frac{1}{2}\right]^{a}_{0}-\left[a_{i}-\frac{1}{2}\right]^{a}_{0}\right)
\langle \tau_{a_{i}+a-1}(\omega) \prod_{k \neq i} \tau_{a_{j}}(1)\rangle^{\mathcal{A}_{1},\mathrm{red}}_{g,1}\\
&+\sum_{j=1}^{s}
\cdot \left[b_{j}+\frac{1}{2}\right]^{a}_{0}
\langle \tau_{b_{j}+a}(\omega) \prod_{i} \tau_{a_{i}}(1)\prod_{k\neq j} \tau_{b_{k}}(\omega)\rangle^{\mathcal{A}_{1},\mathrm{red}}_{g,1}\\
&+\sum_{m}(-1)^{m}\left[-m-\frac{1}{2}\right]^{a}_{0}
\langle \tau_{m}(\omega)\tau_{k-m-1}(\omega) \prod_{i=1}^{r} \tau_{a_{i}}(1) \prod_{j=1}^{s} \tau_{b_{j}}(\omega)
\rangle^{\mathcal{A}_{1},\mathrm{red}}_{g,1}
\end{align*}
\end{prop}

\begin{proof}
Following \cite{localcurves}, we use the term level $k$ to refer to the
Gromov-Witten theory of $\mathbf{P}^1$ with obstruction bundle insertion 
$$R\pi_{\ast}\mathrm{ev}^{\ast}\mathcal{O}(k)$$
on the target.  We will only consider levels $0$, $-1$, and $-2$, where in the first two cases we consider
the top Chern class of the obstruction bundle but in the last case we
take the penultimate Chern class to recover the reduced theory.  In our notation, the level will be 
indicated by superscripts above the brackets.

In order to prove the above relation for level $-2$, we can degenerate $\mathbf{P}^1$ to two rational curves glued at node,
each with normal bundle of degree $-1$.  While we do not have Virasoro rules for $\mathbf{P}^1$ relative to a point, we can continue the reduction process by writing these invariants in terms of absolute theory of $\mathbf{P}^1$ at levels $-1$ and $0$.

The invariants at level $0$ and level $-1$ can be treated as (non-reduced) Gromov-Witten invariants of toric projective Fano surfaces.  In the case of level $-1$, consider the class $E$ of the exceptional divisor 
on the blowup 
$$Y_1 = \operatorname{Bl}_{p}\mathbf{P}^2.$$ 
The Gromov-Witten invariants of $Y_1$ along multiples $E$
are precisely the level $-1$ invariants with the point class 
$\omega$ replaced by the insertion $-[E]$.  

For level $0$, consider a fiber class $F$ along one of the rulings 
of the surface 
$$Y_2 = \mathbf{P}^1 \times\mathbf{P}^1.$$  
We consider Gromov-Witten invariants of $Y_2$ along
multiples of $F$. If we include a point insertion $\tau_{0}(p)$, then an obstruction bundle computation shows
that level $0$ invariants are the same as Gromov-Witten invariants of $Y_2$ with 
$\tau_{0}(p)$ added and with $\omega$ replaced by fiber classes $[C]$ in the other ruling.  
The Virasoro conjecture for $Y_1$ and $Y_2$ has been proven
by Givental \cite{givental}.  If we specialize to our situation, we obtain the following rules for how to remove $\tau_{a+1}(1)$
insertions for degree 1 theories at level $0$ and level $-1$.

In our formulas, we require the function
$$F(c_{1},\dots,c_{N})=\int_{\mathcal{M}_{h,N}}\lambda_h\lambda_{h-1}\psi_{1}^{c_{1}}\dots\psi_{N}^{c_{N}}$$
where the genus $h$ is determined by dimension constraints.  This can be evaluated in terms of Bernoulli numbers
but that is not necessary for our purposes. 

\textbf{Level $0$:}
\nopagebreak
\begin{align*}
\langle \tau_{a+1}(1)&\prod_{i=1}^{r}\tau_{a_{i}}(1)\rangle^{(0)}_{g,1}\cdot \left[\frac{1}{2}\right]^{a}_{0}
=\\
&\left( -2\left[\frac{1}{2}\right]^{a}_{1}+(2a+2)\left[\frac{1}{2}\right]^{a}_{0}\right)
\langle \tau_{a}(\omega) \prod_{i=1}^{n} \tau_{a_{i}}(1)\rangle^{(0)}_{g,1}
\\&+\sum_{i=1}^{r} \left[a_{i}-\frac{1}{2}\right]^{a}_{0}
\langle \tau_{a_{i}+a}(1) \prod_{j \neq i} \tau_{a_{j}}(1)\rangle^{(0)}_{g,1}\\
&+\left(2\left[a_{i}-\frac{1}{2}\right]^{a}_{1}+\left[a_{i}-\frac{1}{2}\right]^{a}_{0}-\left[a_{i}+\frac{1}{2}\right]^{a}_{0}\right)
\langle \tau_{a_{i}+a-1}(\omega) \prod_{j \neq i} \tau_{a_{j}}(1)\rangle^{(0)}_{g,1}\\
&+ \sum_{S\sqcup T = [r]} \sum_{m} (-1)^{m+1}\left[-m-\frac{1}{2}\right]^{a}_{0}\cdot2\cdot F(m, a_{i}\in S)\cdot
 \langle \tau_{a-m-1}(\omega) \prod_{j\in T} \tau_{a_{j}}(1)\rangle^{(0)}_{h,1}
\end{align*}
\textbf{Level $-1$:}
\nopagebreak
\begin{align*}
\langle \tau_{a+1}(1)&\prod_{i=1}^{r}\tau_{a_{i}}(1) \prod_{j=1}^{s} \tau_{b_{j}}(\omega)
\rangle^{(-1)}_{g,1}\cdot \left[\frac{1}{2}\right]^{a}_{0}
=\\ 
&\left[\frac{1}{2}\right]^{a}_{1}
\langle \tau_{a}(\omega) \prod_{i=1}^{r} \tau_{a_{i}}(1)\prod_{j=1}^{s} \tau_{b_{j}}(\omega)\rangle^{(-1)}_{g,1}\\
&+\sum_{i=1}^{r} \left[a_{i}-\frac{1}{2}\right]^{a}_{0}
\langle \tau_{a_{i}+a}(1) \prod_{k \neq i} \tau_{a_{k}}(1)\prod_{j=1}^{m} \tau_{b_{j}}(\omega)\rangle^{(-1)}_{g,1}\\
&-\left[a_{i}-\frac{1}{2}\right]^{a}_{1}
\langle \tau_{a_{i}+a-1}(\omega) \prod_{k \neq i} \tau_{a_{j}}(1)\prod_{j=1}^{m} \tau_{b_{j}}(\omega)\rangle^{(-1)}_{g,1}\\
&+\sum_{j=1}^{s}
\left[b_{j}+\frac{1}{2}\right]^{a}_{0}
\langle \tau_{b_{j}+a}(\omega) \prod_{i} \tau_{a_{i}}(1)\prod_{k\neq j} \tau_{b_{k}}(\omega)\rangle^{(-1)}_{g,1}\\
&+ \sum_{S\sqcup T = [r]} \sum_{m} (-1)^{m+1}\left[-m-\frac{1}{2}\right]^{a}_{0}\cdot F(m, a_{i}\in S)\cdot
 \langle \tau_{a-m-1}(\omega) \prod_{k\in T} \tau_{a_{k}}(1)\prod_{j=1}^{s} \tau_{b_{j}}(\omega)\rangle^{(-1)}_{h,1}\\
&+\frac{1}{2}\sum_{m}(-1)^{m}\left[-m-\frac{1}{2}\right]^{a}_{0}
\langle \tau_{m}(\omega)\tau_{a-m-1}(\omega) \prod_{i=1}^{r} \tau_{a_{i}}(1) \prod_{j=1}^{s} \tau_{b_{j}}(\omega)
\rangle^{(-1)}_{g-1,1}
\end{align*}

The next step is to obtain removal rules for the relative Gromov-Witten theory of $(\mathbf{P}^1, 0)$ at levels $0$ and $-1$.  Again, in
general there are no known Virasoro constraints for relative Gromov-Witten invariants.  We only derive them here for the case we need, namely degree $1$,
using the degeneration formula.  The key feature is that, since we are in degree $1$, 
there is only one possible relative condition, so we only sum over distributions of non-stationary 
insertions to the two possible components.

First, for level 0 relative invariants, we can degenerate the level 0 absolute theory into two copies of the level 0 relative theory. For example,
when there is a single insertion, it is easy to see that 
$$ \langle \tau_{a+1}(1) \rangle ^{(0)}_{g,1} = 2 \langle \tau_{a+1}(1) \rangle^{(0),\mathrm{rel}}_{g,1}.$$
The Virasoro rule for the level $0$ absolute theory implies the following rule for the level $0$ relative theory.

\textbf{Level $0$, relative:}
\nopagebreak
\begin{align*}
\langle \tau_{a+1}(1) \prod_{i=1}^{r} &\tau_{a_{i}}(1)\rangle^{(0),\mathrm{rel}}_{g,1}\cdot 2\cdot\left[\frac{1}{2}\right]^{a}_{0}
=\\
&\left( -2\left[\frac{1}{2}\right]^{a}_{1}+(2a+2)\left[\frac{1}{2}\right]^{a}_{0}\right)
\langle \tau_{a}(\omega) \prod_{i=1}^{r} \tau_{a_{i}}(1)\rangle^{(0),\mathrm{rel}}_{g,1} \\
&+\sum_{i=1}^{r} \left[a_{i}-\frac{1}{2}\right]^{a}_{0}
\langle \tau_{a_{i}+a}(1) \prod_{j \neq i} \tau_{a_{j}}(1)\rangle^{(0),\mathrm{rel}}_{g,1} \\
&+\left(2\left[a_{i}-\frac{1}{2}\right]^{a}_{1}+\left[a_{i}-\frac{1}{2}\right]^{a}_{0}-\left[a_{i}+\frac{1}{2}\right]^{a}_{0}\right)
\langle \tau_{a_{i}+a-1}(\omega) \prod_{j \neq i} \tau_{a_{j}}(1)\rangle^{(0),\mathrm{rel}}_{g,1}\\
&+ \sum_{S\sqcup T = [r]} \sum_{m} (-1)^{m+1}\left[-m-\frac{1}{2}\right]^{a}_{0}\cdot 2\cdot F(m, a_{i}|i\in S)\cdot
 \langle \tau_{a-m-1}(\omega) \prod_{j\in T} \tau_{a_{j}}(1)\rangle^{(0),\mathrm{rel}}_{h,1}
\end{align*}

We have only written the case where all insertions are nonstationary, because that is all that is needed for our purposes.

For the level $-1$ relative theory, we can degenerate 
the level $-1$ absolute theory into the level $-1$ relative theory and the level $0$ relative theory.  As we sum over distributions of marked points, if the $\tau_a(1)$ insertion is assigned to the level 0 component, then we already have determined how to remove it.  As an example, we see that
$$\langle\tau_{a+1}(1)\rangle^{(-1),\mathrm{rel}}_{g,1} = \langle\tau_{a+1}(1)\rangle^{(-1)}_{g,1}
- \langle \tau_{a+1}(1) \rangle^{(0),\mathrm{rel}}_{g,1},$$
which implies a removal rule for the left-hand side.
Again, using the Virasoro rule for level $-1$ invariant and the level $0$ relative invariants we obtain the following rule.

\textbf{Level $-1$, relative:}
\nopagebreak
\begin{align*}
\langle \tau_{a+1}(1) \prod_{i=1}^{r} &\tau_{a_{i}}(1) \prod_{j=1}^{s} \tau_{b_{j}}(\omega)
\rangle^{(-1),\mathrm{rel}}_{g,1}\cdot 2 \cdot \left[\frac{1}{2}\right]^{k}_{0}
=\\
&-(2a+2)
\langle \tau_{a}(\omega) \prod_{i=1}^{r} \tau_{a_{i}}(1)\prod_{j=1}^{s} \tau_{b_{j}}(\omega)\rangle^{(-1),\mathrm{rel}}_{g,1}\\
&+\sum_{i=1}^{r} 2 \cdot \left[a_{i}-\frac{1}{2}\right]^{a}_{0}
\langle \tau_{a_{i}+a}(1) \prod_{k \neq i} \tau_{a_{k}}(1)\prod_{j=1}^{s} \tau_{b_{j}}(\omega)\rangle^{(-1),\mathrm{rel}}_{g,1}\\
&+\left(\left[a_{i}+\frac{1}{2}\right]^{a}_{0}-\left[a_{i}-\frac{1}{2}\right]^{a}_{0}\right)
\langle \tau_{a_{i}+a-1}(\omega) \prod_{k \neq i} \tau_{a_{k}}(1)\rangle^{(-1),\mathrm{rel}}_{g,1}
\\ &+\sum_{j=1}^{s}
2\cdot \left[b_{j}+\frac{1}{2}\right]^{a}_{0}
\langle \tau_{b_{j}+a}(\omega) \prod_{i} \tau_{a_{i}}(1)\prod_{k\neq j} \tau_{b_{k}}(\omega)\rangle^{(-1),\mathrm{rel}}_{g,1}\\
&+ \sum_{m}(-1)^{m}\left[-m-\frac{1}{2}\right]^{a}_{0}
\langle \tau_{m}(\omega)\tau_{a-m-1}(\omega) \prod_{i=1}^{r} \tau_{a_{i}}(1) \prod_{j=1}^{s} \tau_{b_{j}}(\omega)
\rangle^{(-1),\mathrm{rel}}_{g-1,1}
\end{align*}

Finally, the level $-2$ absolute theory - our main objective - can be degenerated into two copies of the level -1 relative theory.  Applying the level -1 relative Virasoro constraint to each side completes the proof.  Fortunately, the function $F$ cancels in the process.
\end{proof}

This concludes the proof of the proposition and also theorem \ref{mainthm}.

\subsection{Generalization to $D,E$ resolutions}\label{dandesection}

We explain here how Theorem \ref{mainthm} can be extended to resolutions
$S_{\Gamma}$ of rational surface singularities associated to root lattices $\Gamma$ of type
$D$ and $E$.  The argument we use here was suggested to us by Jim Bryan, motivated by
a similar argument from \cite{bryan-katz-leung}.  As these singularities are not toric,
there is only a $\mathbb{C}^{*}$-action on $S_{\Gamma}$ and localization techniques are not effective.

The main construction here is to study the versal deformation space of the singularity associated to $\Gamma$ and, via Brieskorn, to study the simultaneous resolution of the universal family.
Let $X_{0} \rightarrow \Delta$ be a smooth family of surfaces over the disk $\Delta$, obtained from a map
from $\Delta$ to the versal deformation space of $S_{\Gamma}$.
While the family is topologically trivial, its fiber over the origin is the resolved surface $S_{\Gamma}$ but all other fibers are given by affine surfaces; in particular, all compact curves on $X_{0}$ lie over the origin.  Again, there is an identification $H_{2}(S_{\Gamma},\mathbb{Z}) = H_{2}(X_{0},\mathbb{Z})=\Gamma$.
This family admits a deformation $X_{z} \rightarrow \Delta$ so that for $z \ne 0$, there are a finite number of non-affine fibers
each isomorphic to $\mathcal{A}_{1}$.  These non-affine fibers are in bijection with positive roots $\alpha$ of $\Gamma$, and the smooth rational curve lies in the corresponding curve class $\alpha$.

An effective curve on $X_{0}$ must be contained in $S_{\Gamma}$ and an effective curve on $X_{z}$ must be contained in one of the copies of $\mathcal{A}_{1}$.  The key observation is that, for
noncontracted curve classes $\beta$, the reduced virtual class on $S_{\Gamma}$
is identical to the relative virtual class of the family $X_{0}$ over $\Delta$"
$$[\Mbar_{g}(S_{\Gamma},\beta)]^{\mathrm{red}} = [\Mbar_{g}(X_{0}/\Delta, \beta)]^{\mathrm{vir}}.$$
The proof of this comparison can be found in \cite{k3paper}.  Similarly, for $X_{z}$,
we have
$$[\Mbar_{g}(\mathcal{A}_{1},\beta)]^{\mathrm{red}} =  [\Mbar_{g}(X_{z}/\Delta, \beta)]^{\mathrm{vir}},$$
where $\beta$ is a multiple of a root curve class and $\mathcal{A}_{1}$ is the corresponding non-affine fiber.  Deformation invariance of the relative virtual class implies that only root curve classes contribute
to $S_{\Gamma}$ and, in that case, the calculation is given by the case of $\mathcal{A}_{1}$.
The result is the following generalization:

\begin{theorem}\label{mainthmde}
 For curve classes of the form $\beta = d \alpha$ and divisors $\omega_{l} \in H^{2}(S_{\Gamma},\mathbb{Q})$, we  have
 \begin{align*}
 \langle \prod_{k=1}^{r} \tau_{a_{k}}(1) &\prod_{l=1}^{s} \tau_{b_{l}}(\omega_{l}) 
 \rangle^{S_{\Gamma}, \mathrm{red}}_{g, d \alpha} =\nonumber\\ 
& \frac{(2g+r+s-3)!}{(2g+s-3)!}
 d^{2g+s-3} 
 \prod_{k=1}^r \frac{(a_k-1)!}{(2a_k-1)!}\left(-\frac{1}{2}
\right)^{a_k-1}\label{mainform}\\
&\cdot\prod_{l=1}^s \frac{b_l!}{(2b_l+1)!}\left(-\frac{1}{2}
\right)^{b_l} (\alpha \cdot \omega_l).\nonumber 
\end{align*}
If $\beta$ is not a multiple of $\alpha$ for any root $\alpha$, then all reduced invariants vanish.
\end{theorem}

\section{Nonrigid $\An\times \mathbf{P}^1$}\label{rubberan}

In this section, we begin to study the full $T$-equivariant theory of the threefold 
$\An\times\mathbf{P}^1$.  The main evaluation of this section involves
invariants associated to a nonrigid target.  We also explain how this geometry 
determines the Gromov-Witten theory of $\mathbf{P}^{1}$.

\subsection{Definitions}

For a curve class $\beta \in H_2(\An, \mathbb{Z})$ and an integer $m\geq0$, we fix two cohomology-weighted partitions $\overrightarrow{\mu}, \overrightarrow{\nu}$ and consider the rubber moduli space
$$\Mbar^{\sim}_{g}(\An \times \mathbf{P}^1, (\beta,m); \mu,  \nu)$$
defined as follows.
This moduli space parametrizes stable maps to a nonrigid target $\An \times \mathbf{P}^1$; that is, two maps are equivalent if they differ by the natural $\mathbb{C}^{\ast}$-scaling action on the $\mathbf{P}^1$
factor.  As before, we require the stable maps to be transverse to the fibers over $0$ and $\infty$, with ramification profiles given by $\mu$ and $\nu$, and to have finite automorphism group with respect to this revised version of equivalence.  In this section, we will be working with \textit{connected} domains and explain how to pass to the disconnected case afterwards.

Rubber invariants of $\An \times \mathbf{P}^1$ are again defined by pulling back cohomology classes via the evaluation maps to the relative divisors and integrating them against the virtual fundamental class.  Because of the $\mathbb{C}^{\ast}$-scaling, the virtual dimension is one less than that of
the usual moduli space of relative stable maps:
$$-1 +2m + (-K_{\An}\cdot\beta) +(l(\mu)- m) + (l(\nu)-m) = l(\mu)+l(\nu) -1.$$

We give an evaluation of the series
\begin{align*}
\langle \overrightarrow{\mu} | \overrightarrow{\nu} &\rangle^{\An, \sim}_{\beta}
 = \sum_{g \geq 0} \langle \overrightarrow{\mu} | \overrightarrow{\nu} \rangle^{\An,\sim}_{g, (\beta,m)}u^{2g}
 \end{align*}
 for nonzero $\beta$.
 For the sake of simplicity, we focus on the case of $\mathcal{A}_{1}$ and explain how to handle the general situation afterwards.

 \subsection{Evaluation for $\mathcal{A}_{1}$}
 
 For the surface $\mathcal{A}_{1}$, the parts of a cohomology-weighted partition are labelled by either $1$ or $\omega$; given a partition $\mu$, we use the notation $\mu(\omega)$ to denote the cohomology-weighted partition where each part is labelled with $\omega$.  We then have the following proposition
 \begin{prop}\label{rubberprop}
For $\beta = d[E]$, we have
$$
\langle \mu(\omega) | \nu(\omega) \rangle^{\mathcal{A}_{1},\sim}_{\beta} 
 = 
 \frac{(t_1+t_2)d^{l(\mu)+l(\nu)-3}}{|\mathrm{Aut}(\mu)|\cdot |\mathrm{Aut}(\nu)|} \frac{\prod \mathcal{S}(d\mu_i u) 
 \prod \mathcal{S}(d\nu_j u)}{\mathcal{S}(du)^2}$$
 where
 $$\mathcal{S}(u) = \frac{sin(u/2)}{u/2}.$$
 If any of parts are labelled by $1$, then the rubber invariant vanishes.
  \end{prop}

We first explain the vanishing statement.  Since $d >0$ and we consider connected domain,
the moduli space is compact so the rubber invariant lies in $\mathbb{Q}[t_1,t_2]$.  Moreover, it must
be divisible by $(t_1+t_2)$, either by arguing via the reduced theory or by expressing the invariant in terms of the $T$-equivariant theory
of the $\mathcal{A}_{1}$-surface as in the next lemma.  Therefore, if nonzero, its cohomological degree is at least $1$.  Since
we have an insertion of degree at most $1$ at each insertion, the maximum possible degree
of the rubber invariant is
$$l(\mu) +l(\nu) - (2m +l(\mu) -m + l(\nu) - m - 1) = 1$$
and equality is achieved if and only if each insertion is labelled by $\omega$.

\subsection{Degree scaling}

We next show that
the degree dependence on $\beta = d[E]$ behaves exactly as in the surface case.   The point is that,
although we are working with the full equivariant theory, only linear terms show up in our calculation.

We will prove the following more general claim.  Consider
any genus $g$ Gromov-Witten invariant on $\mathcal{A}_{1} \times \mathbf{P}^1$, either absolute, relative to one of the divisors
 $A_{1}\times {0}$ and $A_{1}\times{\infty}$, or relative to both divisors.  In the latter case, we allow either
rubber or non-rubber invariants.  Moreover, assume the dimensions of our insertions are such that the 
invariant is linear in $t_1,t_2$, which then forces it to be proportional to $(t_1+t_2)$.  A cohomology class at an insertion is called 
stationary if it is either $\omega$ at a relative marked point
or $\omega$ or $\iota_{\ast}\omega$ at a non-relative marked point, where $\iota:\mathcal{A}_{1}\hookrightarrow\mathcal{A}_{1}\times\mathbf{P}^{1}$ is the inclusion of a fiber.

\begin{lemma}
Given a Gromov-Witten invariant on $\mathcal{A}_{1}\times \mathbf{P}^1$ of the type just discussed, if $s$ is the total number of stationary insertions, then
the invariant is proportional to $d^{2g+s-3}$ as a function of $d$.
\end{lemma}
\begin{proof}
We can show this using the machinery of \cite{TVGW}.  In that
paper, a systematic procedure is given for reducing Gromov-Witten invariants on
$X \times \mathbf{P}^1$ to the Gromov-Witten theory of $X$.  To establish the
degree dependence, we will show it is preserved by each step of the algorithm.  More precisely,
each step of the algorithm is given by a certain relation among invariants
on $\mathcal{A}_{1}$ and $\mathcal{A}_{1}\times\mathbf{P}^{1}$, each of which can be viewed
as a function of $d$.  Although both the genus and number of stationary terms will vary among elements of the relation, we will show that, up to terms we can ignore, the quantity $d^{2g+s-3}$ is fixed among all terms in the relation.  In particular, if every term but one is proportional to $d^{2g+s-3}$, then this implies
that the remaining term is also proportional to $d^{2g+s-3}$. The endpoint of
the algorithm is the Gromov-Witten theory of $\mathcal{A}_{1}$ where we have already proven the 
correct $d$-dependence.

There are three moves involved in the reconstruction result.
\begin{itemize}
\item \textit{Rigidification:}

In this step, nonrigid invariants are expressed in
terms of rigid relative GW-invariants.  In order to do this, we add a
$\tau_0(\omega)$ insertion using the divisor equation and fix it in the $\mathbf{P}^1$ direction
to obtain a non-rubber relative invariant:
\begin{align*}
d \langle \overrightarrow{\mu}| \gamma| \overrightarrow{\nu}\rangle^{\sim}_{g,d}
&= \langle \overrightarrow{\mu}| \gamma\cdot \tau_{0}(\omega)| \overrightarrow{\nu}\rangle^{\sim}_{g,d}\\
&=\langle \overrightarrow{\mu}|\gamma\cdot \tau_{0}(\iota_{\ast}\omega)|\overrightarrow{\nu}\rangle_{g,d}.
\end{align*}
We pick up a factor of $d$ from the divisor equation and increase the number of stationary insertions by $1$, leaving the genus unchanged.  Therefore, $d^{2g+s-3}$ is fixed.
%what about cotangent lines and the divisor equation?%

\item \textit{Degeneration:}

In this step, we have either a relative or absolute invariant and degenerate the $\mathbf{P}^1$-bundle
into two components along the $\mathbf{P}^1$ direction.  A typical relation obtained in this way has the schematic form:
$$
\langle \gamma \rangle_{g,d} = \sum_{\Gamma} \langle \gamma_{1}|\overrightarrow{\mu}\rangle_{\Gamma_{1}}\langle \overrightarrow{\mu^{\vee}}| \gamma_{2}\rangle_{\Gamma_{2}},
$$
where $\Gamma$ denotes the combinatorial configurations
of the degeneration of the domain into connected components.  The degeneration formula
includes a sum over partitions $\mu$, giving relative conditions for each irreducible component of the degeneration, along with a sum over Poincare-dual classes at each relative point.
In our case, this sum over Poincare-dual classes involves one of the following splittings:
$$(1, 2t_1t_2), (\omega, -2\omega).$$

Since we are only interested in invariants proportional
to $(t_1+t_2)$, we can extract the linear part of this relation and ignore any terms that are divisible by $(t_1+t_2)^2$ or which have degree
at least 2 as a rational function in $t_1,t_2$.  With this in mind, the only allowed
degeneration configurations have the following structure.  There is
only one component $C$ which maps nontrivially to $\mathcal{A}_{1}$ since
every such component contributes a factor of $(t_1+t_2)$.    Since the contribution
of this component to the degeneration formula has degree at least $1$, if we remove this primary
component, the remaining components components contribute degree $0$.  The term associated to each combinatorial configuration can be factored into connected components
$$I^{\mathrm{primary}}_{d}\cdot \prod I^{\mathrm{contracted}}_{0},$$
so that only the contribution of $C$ has a nontrivial $d$-dependence.
The degree of the primary term is at least $1$, since it is compactly supported and divisible by $t_1+t_2$.  The degree of each connected contracted terms is at least $0$, with equality if and only if the degree
of all insertions on the connected component equals $2$.  Moreover, any connected component of degree $0$ over $\mathcal{A}_{1}$ must have genus $0$.  Otherwise, there is a contribution of $c_{1}(\mathcal{A}_{1}) = (t_1+t_2)$ to the obstruction bundle.  This restricts
the possibilities as follows.

First, we can have a tree of rational curves connected to $C$ at a single node, with at most one 
stationary marked point from the original insertions $\gamma$.  If there are more than one stationary marked points on the tree, the contribution of the tree will be degree $\geq 1$.  If it contains no stationary marked point, then the Poincare splitting condition at the node must be 
$$(1, 2t_1t_2).$$  
If it contains one stationary marked point from $\gamma$, then the 
Poincare splitting condition at the node is forced to be 
$$(\omega, -2\omega),$$
so $C$ has a new relative stationary insertion.
In either case, genus and the number of stationary insertions on $C$ are unchanged.

The second case is to allow a tree of rational curves connected to $C$ at two nodes.  In this case,
the tree cannot contain any stationary insertions for degree reasons, and the Poincare splitting at each node must be $(\omega, -2\omega)$.  Therefore the genus of the main component has decreased by $1$ but the number of stationary insertions has increased by $2$ so again the $d$-dependence is preserved.

\item \textit{Localization}
The third move in the reconstruction result is virtual localization in the relative and absolute setting.
As in the case of the degeneration step, there is a sum of combinatorial configurations which can be
separated into a primary component with nontrivial $d$-dependence and contracted terms.
A similar argument to the previous one shows that the number of stationary insertions on the
primary component is preserved here as well.  This concludes the proof.
\end{itemize}
\end{proof}

\subsection{Degree $1$ Evaluation}

We now finish the proof of proposition \ref{rubberprop}.  By the previous lemma, we
can assume $d=1$.
\begin{proof} 
 As in proposition \ref{degdep}, we can replace relative insertions 
 labelled with $\omega$with labels of $E_{0}$.
We then apply virtual localization with respect
to the $T$-action on $\mathcal{A}_{1}$.  Because of the $\mathbb{C}^{\ast}$-scaling,
fixed loci on the rubber moduli space can typically be quite complicated to describe.
However, because we consider $d=1$, the situation is much simpler.  

Suppose we have a curve
mapping to $\mathcal{A}_{1}\times\mathbf{P}^1$ that is fixed under
the $T$-action after a possible rescaling in the $\mathbf{P}^1$ direction.
If an irreducible component is contracted under the projection to $\mathcal{A}_{1}$, then it can be an arbitrary curve
mapping to either ${p_1} \times \mathbf{P}^1$ or ${p_2}\times \mathbf{P}^1$.  
Since $d=1$, there is exactly one irreducible component that is
not contracted and it must map isomorphically to $E \subset \mathcal{A}_{1}$.  We can
view this map as the graph of a morphism 
$$f:E \rightarrow \mathbf{P}^1$$ 
defined up to scalar and
identify $E$ with $\mathbf{P}^1$ so that $p_1$ and $p_2$ are identified with $0$ and $\infty$.  Under
these identifications, the condition that our component is $T$-fixed implies that the morphism $f$
is of the form 
$$z \mapsto z^k$$ 
for some integer $k$.  If $k >0$ then
$f$ intersects  $\mathcal{A}_{1}\times {0}$ at $p_1 \times 0$ with multiplicity $k$
and intersects $\mathcal{A}_{1} \times \infty$ at $p_2 \times \infty$ with multiplicity $k$.  We have the same
analysis if $k <0$ with $p_1$ and $p_2$ reversed.  If $k=0$, then this component does not intersect either relative divisor.

Our relative insertions are all at $p_1$, so the only possible choice for this non-contracted component
is $k=0$.  Moreover the remaining components can only be non-contracted over $p_1$ since
otherwise they would intersect the relative divisors over $p_2$.  As a result, the only allowed fixed
loci have the following structure.  For $g_1+g_2=g$, we have a curve of genus $g_1$ that maps
to $p_0 \times \mathbf{P}^1$ with degree $m$ and a curve of genus $g_2$ that is contracted
over $p_1 \times \mathbf{P}^1$ that are connected by
a rational curve mapping to a fiber of the projection to $\mathbf{P}^1$.  We
can rigidify the $\mathbb{C}^{\ast}$-scaling by requiring the connecting edge to map to a 
fixed point of $\mathbf{P}^1$.  As a result, the fixed locus just described is 
$$\Mbar_{g_1}(\mathbf{P}^{1}/0,\infty;\mu, \nu)\times \Mbar_{g_2,1}.$$

\begin{figure*}[htp]
\centering
\psfrag{a}{$\mathcal{A}_{1}$}
\psfrag{p}{$\mathbf{P}^1$}
\psfrag{m}{$\mu$}
\psfrag{n}{$\nu$}
\includegraphics[scale=0.50]{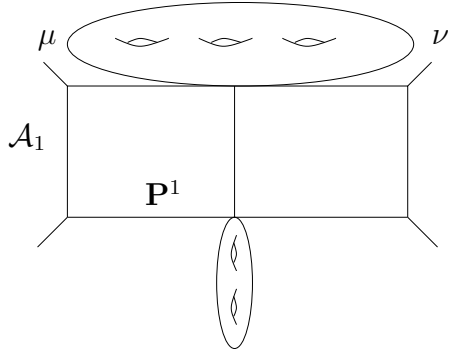}
\caption{Rubber localization with $E_{0}$ insertions}
\end{figure*}

The analysis here is very similar to the localization analysis of proposition \ref{degdep}.
The edge term contributes a factor of $(t_1+t_2)$ so the remaining terms can be calculated
modulo $(t_1+t_2)$.  The contribution from the first factor of the fixed locus is
$$\langle \mu | \Lambda^{\vee}(2t_1) \Lambda^{\vee}(t_2-t_1) \frac{\omega}{(t_2-t_1) - \psi}|\nu\rangle^{\mathbf{P}^1}_{g_1,m},$$ which
is
$$ (-1)^g\langle \mu | \tau_{2g-2 + l(\mu)+l(\nu)}(\omega) | \nu\rangle^{\mathbf{P}^1}_{g,m}$$
modulo $(t_1+t_2)$ since the Hodge classes will cancel by Mumford's relation.

This latter expression is the main calculation in the stationary theory of $\mathbf{P}^1$, consisting
of those Gromov-Witten invariants with only descendents of $\omega$.  These invariants
have been computed in \cite{okpanp1}.  In terms of the trigonometric function
$$\mathcal{S}(u) = \frac{sin(u/2)}{u/2},$$ 
the generating function encoding the stationary theory of $\mathbf{P}^1$ is
\begin{equation}
\sum_{g \geq 0} (-1)^g\langle \mu | \tau_{r}(\omega) | \nu\rangle^{\mathbf{P}^1}_{g,m}u^{2g} = \frac{1}{|\mathrm{Aut}(\mu)||\mathrm{Aut}(\nu)|}\frac{\prod \mathcal{S} (\mu_i u) \prod \mathcal{S}(\nu_j u)}{\mathcal{S}(u)}.\label{P1formula}
\end{equation}
In the above generating function, the index $r$ of the insertion is fixed by the dimension constraint
to be $2g-2 + l(\mu)+l(\nu)$.

Similarly, the contribution of the second factor of the fixed locus is the coefficient of $u^{2g_2}$ in
$$\sum_{g \geq 0} (-1)^g u^{2g}\int_{\Mbar_{g,1}} \lambda_g \psi^{2g-2} = \frac{1}{\mathcal{S}(u)}.$$
This evaluation has been computed in \cite{faber-pandharipande}.

Combining the two generating functions gives the answer.
\end{proof}

\subsection{Stationary theory of $\mathbf{P}^1$}\label{stationarysection}

The rubber evaluation was derived using the stationary theory of $\mathbf{P}^1$ in (\ref{P1formula}).  
However, another choice of insertions in our evaluation gives an answer in terms of certain
double Hurwitz numbers.  As these double Hurwitz numbers are simple to calculate directly, 
this gives a new derivation of the stationary theory of $\mathbf{P}^1$.  From that specific expression, it is possible to derive the stationary theory of target curves of arbitrary genus $h$ by degeneration to a nodal configuration of rational curves.  In 
particular, the stationary theory of $\mathbf{P}^1$ directly yields the Gromov-Witten/Hurwitz correspondence of \cite{okpanp1}.  In \cite{okpanp2}, the original derivation requires understanding the full equivariant theory of $\mathbf{P}^1$.  While that approach is more involved than this one, it is of course a much stronger result.

Given two partitions $\rho,\lambda$ of $m$, let 
$$H^{g}_{\rho,\lambda}$$ 
denote the number of disconnected genus $g$ covers of $\mathbf{P}^1$ with a
branch point of ramification profile $\rho$, a branch point of ramification profile $\lambda$,
and simple ramification everywhere else.

Consider the double Hurwitz series
$$H_{\rho,\lambda}(u) = \sum_{g} \frac{m^{1-r}}{r!} H^{g}_{\rho,\lambda} u^{2g},$$
where $r= 2g-2+l(\lambda)+l(\rho)$ and $m = |\rho|=|\lambda|$.
By comparing two evaluations of the rubber $\mathcal{A}_{1}$ theory, we have the following
proposition for the stationary theory of $\mathbf{P}^1$.

\begin{prop}
\begin{align*}
\sum_{g \geq 0} (-1)^g\langle \mu | \tau_{2g-2 + l(\mu)+l(\nu)}(\omega) | \nu\rangle^{\mathbf{P}^1}_{g,m}u^{2g}=
\frac{1}{|\mathrm{Aut}(\mu)||\mathrm{Aut}(\nu)|}H_{\mu,(m)}(u)\cdot H_{\nu,(m)}(u)\cdot \mathcal{S}(u).
\end{align*}
\end{prop}
\begin{proof}
As before, we calculate the degree $1$ rubber invariant
$$\langle \mu(\omega)| \nu(\omega) \rangle^{\mathcal{A}_{1},\sim}_{1},$$
where $\mu$ and $\nu$ are partitions of $m$.
However, we now replace the insertions at $\mu$ with $E_0$ and the insertions at $\nu$ with $E_2$ 
before applying virtual localization.  The fixed loci can be analyzed as before.  With our new choice of insertions, the intersection with the relative divisor $\mathcal{A}_{1} \times 0$ is entirely over $p_0$
and the inersection with $\mathcal{A}_{1} \times \infty$ is entirely over $p_1$.  As a consequence,
the curve component that is not contracted by the projection to $\mathcal{A}_{1}$ must correspond
to the graph of 
$$f:z \mapsto z^m$$ 
in our previous notation.  

Up to genus distribution, there is
a unique configuration that allows this.  The target degenerates into three pieces.
In the first piece, we have a genus $g_1$ curve mapping to $p_1 \times \mathbf{P}^1$
with ramification $\mu$ over $0$ and ramification $(m)$ over $\infty$.  In the central piece, we have 
the rational curve that is not contracted by the projection to $\mathcal{A}_{1}$; its ramification profile is $(m)$ 
over each relative divisor.  Finally, in the third piece, we have a genus $g_2$ curve mapping
to $p_2 \times \mathbf{P}^1$ with ramification $(m)$ over $0$ and ramification $\nu$ over
$\infty$. Stable maps to each piece are still defined only up to a $\mathbb{C}^{\ast}$-scaling.

\begin{figure*}[htp]
\centering
\psfrag{a}{$\mathcal{A}_{1}$}
\psfrag{p}{$\mathbf{P}^1$}
\psfrag{m}{$\mu$}
\psfrag{n}{$\nu$}
\includegraphics[scale=0.60]{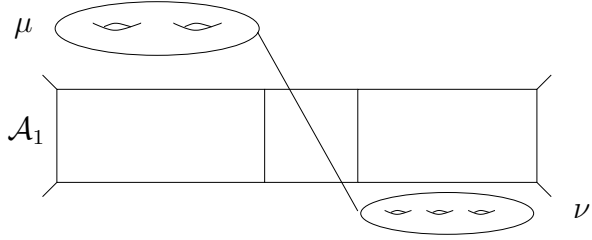}
\caption{Rubber localization with $E_{0}$ and $E_{2}$ insertions}
\end{figure*}

As this fixed locus features a degenerate target, there are now cotangent lines $\psi_{0},\psi_{\infty}$ in the virtual
normal bundle that correspond to smoothing the target.  See \cite{TVGW} for a careful description
of these and how to remove them.  As before, we work mod $(t_1+t_2)^2$; let $\tau = (t_2-t_1)$
then the contribution is given by

$$(t_1+t_2)\cdot \langle \mu| \frac{1}{m\tau - \psi_\infty}|(m)\rangle^{\mathbf{P}^1,\sim}_{g_1}\cdot m^2
\cdot \langle(m) | \frac{1}{-m\tau - \psi_0}|\nu\rangle^{\mathbf{P}^1,\sim}_{g_2}.$$
As always, the Hodge classes cancel by Mumford's relation. 

The following lemma is well-known (see \cite{llz} and \cite{okpanp1} for example); we sketch a brief
justification below.

\begin{lemma}\label{rubberhurwitz}
$$\langle \rho| \frac{1}{1- \psi_\infty}|\lambda \rangle^{\mathbf{P}^{1},\sim}_{g}= \frac{1}{|\mathrm{Aut}(\rho)
\mathrm{Aut}(\lambda)|}\frac{H^{g}_{\rho,\lambda}}{(2g-2+l(\lambda)+l(\rho))!}$$
\end{lemma}
The expression in the factorial is the number of simple ramification points.
\begin{proof}
If we sum both sides over genus and view the result as operators on the space of partitions, both sides satisfy a differential equation of the form
$$u\frac{\partial}{\partial u} \mathsf{S} = \mathsf{M}\mathsf{S},$$
where $\mathsf{M}$ is the cut-and-join operator in Hurwitz theory.  On the left-hand side, this
follows by rigidifying the rubber geometry with a dilaton insertion $\tau_{1}(1)$ and removing
the cotangent lines $\psi_\infty$ with topological recursion relations.  On the right-hand side, this
is follows from picking a simple ramification point and degenerating it onto a separate component.  Since
the lowest-order terms match, this forces the entire series to agree.
\end{proof}

If we compare this expression in Lemma \ref{rubberhurwitz} with the original choice of insertions from the last section, we immediately
have the proposition.
\end{proof}

When one of the partitions is totally ramified, i.e. $\lambda = (m)$, then
double Hurwitz numbers can be simply evaluated using the character theory 
of $S_m$.  This computation has been performed in \cite{GJV}.  

\begin{prop}
$$H(\rho, (m))(u) = \frac{\prod_{i=1}^{l(\rho)}\mathcal{S}(\rho_{i}u)}{\mathcal{S}(u)}$$
\end{prop}

If we combine these two propositions, we obtain a new proof of equation (\ref{P1formula}).  

\subsection{Extension to $\An$}

The nonrigid theory of $\An\times\mathbf{P}^1$ reduces to the case of $\mathcal{A}_1\times \mathbf{P}^1$ by reducing to the surface calculations in a manner identical to the proof of degree scaling.
 Given two cohomology-weighted partitions $\overrightarrow{\mu},\overrightarrow{\nu}$ labelled with cohomology classes $\gamma_1, \dots,\gamma_{l(\mu)}$ and $\eta_1, \dots,
\eta_{l(\nu)}$ respectively.  

\begin{prop}\label{rubberpropan}
If $\beta = d\alpha$ for a root curve class $\alpha$ and each label $\gamma_{i}, \eta_{j}$ is
a divisor, we have
$$
\langle \overrightarrow{\mu} | \overrightarrow{\nu}\rangle^{\An}_{\beta} 
 = 
 \frac{(t_1+t_2)d^{l(\mu)+l(\nu)-3}}{|\mathrm{Aut}(\mu)|\cdot |\mathrm{Aut}(\nu)|} \frac{\prod_{i} (\alpha\cdot\gamma_{i})\mathcal{S}(d\mu_i u) 
 \prod_{j}(\alpha\cdot\eta_{j}) \mathcal{S}(d\nu_j u)}{\mathcal{S}(du)^2}$$
 where
 $$\mathcal{S}(u) = \frac{sin(u/2)}{u/2}.$$
 Otherwise, the series vanishes.
\end{prop} 
After specializing to $t_{1}=t_{2}$, the same statement holds for $D,E$ resolutions.

\section{Relative invariants of $\An\times\mathbf{P}^1$}\label{relan}

In this section, we study the relative Gromov-Witten theory of $\An\times \mathbf{P}^1$.
The results of the last section allow us to calculate
$$\mathsf{Z}^{\prime}(\An\times\mathbf{P}^1)_{\overrightarrow{\mu},\overrightarrow{\rho},\overrightarrow{\nu}} 
\in \mathbb{Q}(t_1,t_2)((u))[[s_1,\dots,s_n]]$$
for
 $$\overrightarrow{\rho} = \{(2,1),(1,1)^{m-2}\}\quad\mathrm{or}\quad\overrightarrow{\rho}=\{(1,\omega_{i}),(1,1)^{m-1}\}.$$
We will
abbreviate these partitions as $(2)$ and $(1,\omega_{i})$ respectively. 
In these cases, we can establish the equivalence between this theory
and the quantum cohomology of $\mathrm{Hilb}(\An)$ computed in \cite{hilban}, where
these partitions correspond to divisors on the Hilbert scheme.

While we are unable to go further, we state a generation conjecture from that paper and prove that 
it implies an algorithm for calculating the full relative series in terms of
\begin{enumerate}
\item the local theory of $\mathbb{C}^{2}\times \mathbf{P}^1$
\item rubber invariants from the last section
\item degeneration techniques.
\end{enumerate}
We sketch the extension of this algorithm to $\An$-bundles over a higher genus curve.

\subsection{Rigidification}
In this section, we will compute
$$\mathsf{Z}^{\prime}(\An\times \mathbf{P}^1)_{\overrightarrow{\mu}, (2),\overrightarrow{\nu}}$$
and 
$$\mathsf{Z}^{\prime}(\An\times \mathbf{P}^1)_{\overrightarrow{\mu}, (1,\omega_{i}),\overrightarrow{\nu}}.$$
Recall that the generating function is defined by allowing possibly \textit{disconnected} domain curves.
If we fix a configuration of connected domain components, the associated Gromov-Witten invariant
is a product of the associated connected Gromov-Witten invariants.  It thus suffices to 
study the restricted generating function for connected domain curves of genus $g$.  That is, we 
compute the partition functions
$$\mathsf{Z}^{\circ}(\An\times\mathbf{P}^1)_{\overrightarrow{\mu},(2),\overrightarrow{\nu}},
\mathsf{Z}^{\circ}(\An\times\mathbf{P}^1)_{\overrightarrow{\mu}, (1,\omega_{i}),\overrightarrow{\nu}},
\mathsf{Z}^{\circ}(\An\times\mathbf{P}^1)_{\overrightarrow{\mu}, (1),\overrightarrow{\nu}},$$
defined using the moduli spaces
$$\Mbar^{\circ}_{g}(\An \times \mathbf{P}^1, (\beta,m); \mu, \rho, \nu)$$
of relative stable maps of \textit{connected} curves of genus $g$ with the appropriate ramification
profile over the relative divisors.  We split the generating function into the contribution from curve classes
$(0,m)$ and curve classes $(\beta, m)$ with $\beta$ nonzero:
$$\mathsf{Z}^{\circ}_{\beta = 0}(\An\times \mathbf{P}^1)_{\overrightarrow{\mu},\overrightarrow{\rho},\overrightarrow{\nu}} + 
\mathsf{Z}^{\circ}_{\beta\neq 0}(\An\times\mathbf{P}^1)_{\overrightarrow{\mu},\overrightarrow{\rho},\overrightarrow{\nu}}.$$  
\begin{prop}
The generating functions
$$\mathsf{Z}^{\circ}_{\beta=0}(\An\times\mathbf{P}^1)_{\overrightarrow{\mu},\overrightarrow{\rho},\overrightarrow{\nu}}$$
are determined in terms of the theory of $\mathbb{C}^2\times\mathbf{P}^1.$  
\end{prop}
\begin{proof}
This follows immediately from $T$-localization along the $\An$-direction.  Since the domain is contracted by the projection to $\An$, the contribution of each fixed locus is given by the associated integral on $\mathbb{C}^2\times\mathbf{P}^1$.
\end{proof}

For the remaining contributions, the case where the relative partition is $(1)^{m}$ is the easiest.
\begin{lem}\label{lemma1}
$$\mathsf{Z}^{\circ}_{\beta\neq 0}(\An\times\mathbf{P}^1)_{\overrightarrow{\mu}, (1), \overrightarrow{\nu}}=0.$$
\end{lem}
\begin{proof}
As $\beta \neq 0$, the invariant is a polynomial in $t_1,t_2$ divisible by $(t_1+t_2)$.  However, from
dimension constraints, the maximum cohomological degree of the invariant occurs when every part of $\mu,\nu$ is labelled with a divisor in which case
this only gives degree $0$.
\end{proof}

For the remaining contribution with $\overrightarrow{\rho} = (2)$ or $\overrightarrow{\rho} = (1,\omega_{i})$, 
we have the following evaluation.
Given two cohomology-weighted partitions $\overrightarrow{\mu},\overrightarrow{\nu}$ labelled with cohomology classes $\gamma_1, \dots,\gamma_{l(\mu)}$ and $ \eta_1, \dots,
\eta_{l(\nu)}$ respectively, let
\begin{align*}
\Theta^{\circ}(\overrightarrow{\mu}, \overrightarrow{\nu})& =\frac{(t_1+t_2)}{|\mathrm{Aut}(
\overrightarrow{\mu})|\cdot |\mathrm{Aut}(\overrightarrow{\nu})|}\cdot\\ 
&\sum_{1\leq i < j\leq n} \sum_{d=1}^{\infty} (du)^{l(\mu)+l(\nu)-2}\frac{
\prod_{k=1}^{l(\mu)}(\alpha_{i,j}\cdot \gamma_{k})\mathcal{S}(d\mu_{k}u)
\prod_{k=1}^{l(\nu)}(\alpha_{i,j}\cdot \eta_{k})\mathcal{S}(d\nu_{k}u)}{d\mathcal{S}(du)^2} 
(s_i\cdot\dots\cdot s_{j-1})^{d}.
\end{align*}
Up to a monomial shift, this is precisely the rubber evaluation from last section.

\begin{prop}\label{relprop}
If $\mu, \nu$ are partitions of $m >0$ and the cohomology classes labelling $\mu, \nu$ are divisors, then we have
$$u^{l(\mu)+l(\nu)-1}\mathsf{Z}^{\circ}_{\beta \neq 0}(\An\times\mathbf{P}^1)_{\overrightarrow{\mu},(2),\overrightarrow{\nu}}
= \frac{d}{du}\Theta^{\circ}(\overrightarrow{\mu},\overrightarrow{\nu})$$
and
$$
u^{l(\mu)+l(\nu)}\mathsf{Z}^{\circ}_{\beta\neq 0}(\An\times\mathbf{P}^1)_{\overrightarrow{\mu},(1,\omega_{k}),\overrightarrow{\nu}}
= s_{k}\frac{d}{ds_k}\Theta^{\circ}(\overrightarrow{\mu},\overrightarrow{\nu}).$$
Otherwise, we have
$$\mathsf{Z}^{\circ}_{\beta \neq 0}(\An\times\mathbf{P}^1)_{\overrightarrow{\mu},(2),\overrightarrow{\nu}}=
\mathsf{Z}^{\circ}_{\beta\neq 0}(\An\times\mathbf{P}^1)_{\overrightarrow{\mu},(1,\omega_{k}),\overrightarrow{\nu}}=0.$$
In particular, after multiplication by a monomial in $u$, these three-point functions 
are rational functions of $e^{iu},s_1,\dots,s_n$.
\end{prop}
\begin{proof}
The vanishing statement for non-divisor insertions follows from dimension constraints.
For the rest, we proceed by applying a rigidification argument to our rubber evaluation.  
For the relative insertion $(2)$, the dilaton equation
allows us to add a $\tau_{1}(1)$ insertion to our rubber invariant.  We can replace the nonrigid invariant with a rigid relative invariant by using this marked point to fix the $\mathbb{C}^{\ast}$ action; that is,
if we impose the condition that this point lies on a fixed $\An$-fiber, we have
\begin{align*}
\langle \overrightarrow{\mu}|\tau_{1}[F]|\overrightarrow{\nu}\rangle^{\circ}_{g,\beta} &= \langle \overrightarrow{\mu}| \tau_{1}(1)|\overrightarrow{\nu}\rangle^{\sim}_{g,\beta} \\&= (2g-2+l(\mu)+l(\nu))\cdot\langle \overrightarrow{\mu}|\overrightarrow{\nu}\rangle^{\sim}_{g,\beta}.
\end{align*}
This last equality is precisely the dilaton equation.  Because of the monomial shift between
$\Theta^{\circ}(\overrightarrow{\mu},\overrightarrow{\nu})$ and our rubber evaluation, the generating function of these rigidified invariants is precisely $u\frac{d}{du}\Theta^{\circ}$.

By degenerating the base $\mathbf{P}^1$, we can arrange to have two components so that our two relative points lie on one component $C_1$ and our fiber insertion lies on the other component, $C_2$.
The degeneration formula gives
\begin{align*}
\langle \overrightarrow{\mu}|\tau_{1}[F]|\overrightarrow{\nu}\rangle^{\circ}_{g,\beta}
&= \sum_{\substack{\overrightarrow{\rho}\\ \beta_{1}+\beta_{2}=\beta\\ \Gamma_{1},\Gamma_{2}}}
\langle \overrightarrow{\mu},\overrightarrow{\nu}, \overrightarrow{\rho}\rangle_{\Gamma_{1},\beta_{1}}
\frac{1}{\mathfrak{z}(\rho)}\langle \overrightarrow{\rho}^{\vee}|\tau_{1}[F]|\rangle_{\Gamma_{2},\beta_{2}}.
\end{align*}
In this equation, we are summing over all possible combinatorial configurations $\Gamma_{1}, \Gamma_{2}$ of connected domain components so that the glued curve
over $C_{1}\cup C_{2}$ is connected.  The notation $\overrightarrow{\rho}^{\vee}$
denotes the partition $\rho$ with Poincare-dual cohomology insertions and the factor
$$\mathfrak{z}(\overrightarrow{\rho}) = \prod_{j} \rho^{(j)}\cdot |\mathrm{Aut}(\overrightarrow{\rho})|$$
is the gluing term from the degeneration formula.

First, if $\beta_{2}\ne 0$, then consider a connected component of $\Gamma_{2}$ which is not contracted under the projection to $\An$.  The dimension constraint from Lemma \ref{lemma1}
again shows that this invariant vanishes, unless the curve is contracted by the projection
to $\mathbf{P}^{1}$ and does not intersect the relative divisor.  However, this violates the constraint that the total degeneration configuration is connected.  As
$\beta_{2}=0$, we can use \cite{localcurves} to evaluate the second factor in the right-hand side.
The only term that contributes is with $\overrightarrow{\rho}=(2)$ and $\Gamma_{1}$ given by a connected curve of genus $g$:
$$\langle \overrightarrow{\mu}|\tau_{1}[F]|\overrightarrow{\nu}\rangle^{\circ}_{g,\beta}
= \langle \overrightarrow{\mu},(2),\overrightarrow{\nu}\rangle^{\circ}_{g,\beta}.$$
Along with the rigidification statement, this completes the proof.

The same argument applies for $(1,\omega_{k})$ with one modification.  Instead of adding a marked point via the dilaton equation, we can use the divisor equation and again use the marked point
to rigidify the rubber scaling:
\begin{align*}
\langle \overrightarrow{\mu}|\tau_{0}(\iota_{\ast}\omega_{k})|\overrightarrow{\nu}\rangle^{\circ}_{g,\beta} &= \langle \overrightarrow{\mu}| \tau_{0}(\omega_{k})|\overrightarrow{\nu}\rangle^{\sim}_{g,\beta} \\&= (\omega_{k}\cdot\beta)\langle \overrightarrow{\mu}|\overrightarrow{\nu}\rangle^{\sim}_{g,\beta}.
\end{align*}
The rest of the argument goes through unchanged.
The fact that derivatives of $\Theta^{\circ}$ are rational functions is an elementary
check.
\end{proof}

We will write down an expression for the disconnected $\beta\ne0$ partition functions in the next section.

\subsection{Fock space}

We introduce the Fock space modelled on
$H^{\ast}_{T}(\An, \mathbb{Q})$.  As we will discuss later, this describes the 
$T$-equivariant cohomology of the Hilbert scheme of points of $\An$.  
By definition, the Fock space $\mathcal{F}_{\An}$ is freely generated over 
$\mathbb{Q}(t_1,t_2)$ by the action of commuting creation operators
$$\mathfrak{p}_{-k}(\gamma)$$
for $k>0$ and $\gamma \in H^{\ast}_{T}(\An,\mathbb{Q})$
on the vacuum vector $v_{\emptyset}$.
The annihiliation operators 
$$\mathfrak{p}_{k}(\gamma)$$
for $k>0$ kill the vacuum vector
$$\mathfrak{p}_{k}(\gamma)\cdot v_{\emptyset} = 0$$
and satisfy the commutation relations
$$ [\mathfrak{p}_{k}(\gamma_{1}), \mathfrak{p}_{l}(\gamma_{2})] = -k \delta_{k+l}(\gamma_1, \gamma_2)
 $$
where $(\gamma_1, \gamma_2)$ denotes the Poincare pairing on $H_{T}^{\ast}(\An, \mathbb{Q})$.
We define a nondegenerate pairing on $\mathcal{F}_{\An}$ by requiring
$$\langle v_{\emptyset}| v_{\emptyset} \rangle = 1$$
and specifying the adjoint
$$\mathfrak{p}_{k}(\gamma)^{\ast} = -\mathfrak{p}_{-k}(\gamma).$$
There is an orthogonal grading
$$\mathcal{F}_{\An} = \bigoplus_{m\geq 0} \mathcal{F}_{\An}^{(m)}$$
induced by defining the degree of $v_{\emptyset}$ to be zero and the degree
of each operator $\mathfrak{p}_{k}(\gamma)$ to be $-k$.

If we work with a fixed basis $\{\gamma_{0}, \dots, \gamma_{n}\}$, our Fock space has a natural basis indexed by cohomology-weighted partitions with labels in our basis.
Given 
$$\overrightarrow{\mu} = \{(\mu_1, \gamma_{i_{1}}), \dots, (\mu_l, \gamma_{i_{l}})\},$$
the associated basis element is given by
$$\frac{1}{\prod \mu_{i}\cdot|\mathrm{Aut} \overrightarrow{\mu}|}
\mathfrak{p}_{-\mu_{1}}(\gamma_{i_{1}})\cdot\dots \mathfrak{p}_{-\mu_{l}}(\gamma_{i_{l}})
\cdot v_{\emptyset}.$$
A basis of the graded piece $\mathcal{F}_{\An}^{(m)}$
is given by cohomology-weighted partitions of $m$.
Under the inner product described before, the dual basis is given by cohomology-weighted partititions labelled with the dual basis of $\{\gamma_{i}\}$.

The first application of this formalism is to handle the combinatorics of the disconnected partition function.  Let
$$
\Theta^{\bullet}(\overrightarrow{\mu},\overrightarrow{\nu}) = 
\sum_{\substack{\overrightarrow{\mu} = \overrightarrow{\mu_{1}}\cup \overrightarrow{\rho}
\\ \overrightarrow{\nu} = \overrightarrow{\nu_{1}}\cup \overrightarrow{\rho}}}
(-1)^{|\rho|-l(\rho)}
\langle \overrightarrow{\rho}| \overrightarrow{\rho}\rangle
\Theta^{\circ}(\overrightarrow{\mu_{1}}, \overrightarrow{\nu_{1}}),
$$
where the brackets denote the Fock space inner product and the summation is over common subpartitions $\overrightarrow{\rho}$ of both $\overrightarrow{\mu}$ and $\overrightarrow{\nu}$.
\begin{prop}\label{relprop2}
\begin{align*}
u^{l(\mu)+l(\nu)-1}\mathsf{Z}^{\prime}_{\beta \neq 0}(\An\times\mathbf{P}^1)_{\overrightarrow{\mu},(2),\overrightarrow{\nu}}
&= \frac{d}{du}\Theta^{\bullet}(\overrightarrow{\mu},\overrightarrow{\nu})\\
u^{l(\mu)+l(\nu)}\mathsf{Z}^{\prime}_{\beta\neq 0}(\An\times\mathbf{P}^1)_{\overrightarrow{\mu},(1,\omega_{k}),\overrightarrow{\nu}}
&= s_{k}\frac{d}{ds_k}\Theta^{\bullet}(\overrightarrow{\mu},\overrightarrow{\nu}).
\end{align*}
\end{prop}
\begin{proof}
We just prove the first case.  It follows from lemma \ref{lemma1} that there can only be one connected component that is not contracted by the projection to $\An$ and it must contain the relative marked point associated to the part $2$.  Let $\overrightarrow{\mu_{1}}, \overrightarrow{\nu_{1}}$ be the 
relative conditions associated to this primary component.  By dimension counting, any other connected component must be a rational curve with maximal ramification degree over $\mathbf{P}^{1}$, totally ramified over the relative divisors corresponding to $\mu$ and $\nu$.  This already implies
that the remaining relative conditions coincide:
$$\overrightarrow{\mu}\backslash\overrightarrow{\mu_{1}}= \overrightarrow{\nu}\backslash\overrightarrow{\nu_{1}}= \overrightarrow{\rho}.$$
It is easy to check that the contribution of these rational curves matches the Fock space inner product up to a sign.
\end{proof}

In \cite{hilban}, this complicated expression is expressed in terms of operators arising from an action of the affine algebra $\widehat{\mathfrak{gl}}(n+1)$ on Fock space.

\subsection{Ring structure}\label{ringstructure}

Let 
$$R = \mathbb{Q}(t_1,t_2)((u))[[ s_1,\dots, s_n]]$$
denote the ring of Laurent series in $u, s_1, \dots, s_n$ with coefficients in $\mathbb{Q}(t_1,t_2)$.
We will use the relative invariants of $\An \times \mathbf{P}^1$ to define the structure of
an $R$-algebra
on
$$\mathcal{R}^{(m)}_{\mathrm{GW}}(\An) =\mathcal{F}_{\An}^{(m)} \otimes_{\mathbb{Q}(t_1,t_2)} R.$$

Given three cohomology-weighted partitions $\overrightarrow{\mu}, \overrightarrow{\nu}, 
\overrightarrow{\rho}$ of $m$, we
define a product $\ast$ using the following structure constants
$$\langle \overrightarrow{\mu}, \overrightarrow{\nu} \ast 
\overrightarrow{\rho} \rangle = (-iu)^{-m+l(\mu)+l(\nu)+l(\rho)}\mathsf{Z}^{\prime}(\An\times\mathbf{P}^1)_{\overrightarrow{\mu}, \overrightarrow{\nu}, 
\overrightarrow{\rho}}$$
and extending by $R$-linearity.

\begin{prop}
Under the product defined above, $\mathcal{R}^{(m)}_{\mathrm{GW}}(\An)$ satisfies the axioms
of an $R$-algebra with $(1,\dots,1)$ as the identity element.
\end{prop}
\begin{proof}
Commutativity is obvious.  The evaluation of the identity element follows from lemma \ref{lemma1}.  For associativity, we
consider $\An\times \mathbf{P}^1$ relative to four points $z_1, \dots, z_4$.  If we degenerate 
$\mathbf{P}^1$ to a broken $\mathbf{P}^1$ with two points on each component, there are two choices
for the distribution of points.  The degeneration formula with respect to 
these two configurations yields the associativity constraint.  The shift of $u$ in the definition of our structure constants ensures that the genus parameters match up correctly.
\end{proof}

Except for the claim about the identity element, this construction of
a ring structure with a basis indexed by cohomology-weighted partitions is valid for any surface $S$.
For most surfaces, e.g. the Enriques surfaces, it is easy to see that the unit element of the deformed
algebra structure must be a nontrivial deformation of $(1)^{m}$.

\subsection{Comparison to Quantum Cohomology of the Hilbert Scheme}

The advantage of rewriting our relative theory in terms of a ring structure on Fock space is
that we can compare it to another such ring structure.  The Hilbert scheme of $m$ points on
$\An$ parametrizes subschemes of length $m$ on the surface $\An$.  The $T$-equivariant 
cohomology of $\mathrm{Hilb}(\An)$, taken over all numbers of points, has a geometric identification with the Fock space
$\mathcal{F}_{\An}.$  The Heisenberg operators are geometrically defined 
using correspondences between Hilbert schemes of different numbers of points, (\cite{G},\cite{N}).
Our distinguished basis corresponds precisely to the Nakajima basis indexed by cohomology-weighted partitions.  Given a cohomology-weighted partitition $(\mu_1,\delta_1),\dots,(\mu_l,\delta_l)$
of $m$, the associated cohomology class
on $\mathrm{Hilb}_{m}(\An)$ has degree
$$2(m - l(\mu)) + \sum \mathrm{deg}(\delta_k).$$
In particular, the partitions $(2,1,\dots,1)$ and $(1,\omega_{k})$
are divisors and give a basis of $H^{2}(\mathrm{Hilb}_{m}(\An),\mathbb{Q})$.
The inner product described matches the classical Poincare pairing
on $T$-equivariant cohomology.

The classical ring structure on $H^{\ast}_{T}(\mathrm{Hilb}_{m}(\An),\mathbb{Q})$ induces a ring structure on 
each graded part $\mathcal{F}_{\An}^{(m)}$ of our Fock space.  
We are interested in the quantum cohomology, which defines a ring structure
on 
$$QH^{\ast}_{T}(\mathrm{Hilb}_m\An)= \mathcal{F}_{\An}^{(m)}\otimes \mathbb{Q}(t_1,t_2)((q))[[s_1,\dots, s_n]].$$
with structure constants determined by rational curves on $\mathrm{Hilb}(\An)$.  The
variables $q$ and $s_1,\dots, s_n$ encode the degree of our curves with
respect to the divisors $-(2)$ and $(1,\omega_{k}), k=1,\dots,n$ respectively.
This ring structure has been computed explicitly in \cite{hilban}.
We denote by
$$\langle\overrightarrow{\mu},\overrightarrow{\nu},\overrightarrow{\rho}\rangle^{\mathrm{Hilb}}_{\An}$$
the structure constants of the quantum cohomology ring with respect to the Poincare pairing.
\begin{prop}
For the divisor class $(2)$ and $(1,\omega_{k})$, the structure constants
$$\langle \overrightarrow{\mu},(2),\overrightarrow{\nu}\rangle^{\mathrm{Hilb}}_{\An}, 
\langle \overrightarrow{\mu},(1,\omega_{k})),\overrightarrow{\nu}\rangle^{\mathrm{Hilb}}_{\An}$$
are explicitly given rational function in $q$ and $s_1,\dots, s_n$.
Under the variable substitution $q = -e^{iu}$,
we have
\begin{align*}
(-1)^{m}\langle\overrightarrow{\mu},(2),\overrightarrow{\nu}\rangle^{\mathrm{Hilb}}_{\An} = 
(-iu)^{-1+l(\mu)+l(\nu)}\mathsf{Z}^{\prime}(\An\times\mathbf{P}^1)_{\overrightarrow{\mu},(2),\overrightarrow{\nu}}
\end{align*}
and
\begin{align*}
(-1)^{m}\langle\overrightarrow{\mu},(1,\omega_{k}),\overrightarrow{\nu}\rangle^{\mathrm{Hilb}}_{\An} = 
(-iu)^{l(\mu)+l(\nu)}\mathsf{Z}^{\prime}(\An\times\mathbf{P}^1)_{\overrightarrow{\mu},(1,\omega_{k}),\overrightarrow{\nu}}.
\end{align*}
\end{prop}

This proposition is proven by a direct computation of the Hilbert scheme three-point invariant, followed by comparison with proposition 
\ref{relprop2}.  This last statement is the Gromov-Witten/Hilbert correspondence for divisor operators.

\subsection{Generation conjecture}\label{generationconjecture}

The following conjecture is presented in \cite{hilban}.
\begin{conj}
For the surface $\An$,
the operators of quantum multiplication by $(2)$ and $(1,\omega_{k})$
have nondegenerate joint spectrum, i.e. their joint eigenspaces are one-dimensional.
\end{conj}

It is an immediate consequence of this conjecture that the divisors generate the quantum cohomology ring for $\mathrm{Hilb}_{m}(\An)$.  The same approach proves that divisors generate the quantum ring
for $\mathrm{Hilb}(\mathbb{C}^2)$.  Unfortunately, while we are unable to prove the conjecture,
we do provide suggestive evidence for its validity.  For the rest of the section, we explain some consequences of this nondegeneracy claim.  The following two corollaries are directly implied by the
above conjecture.

\begin{corstar}\label{gencor}
Assuming the generation conjecture for the surface $\An$,
the partitions $(2)$ and $(1,\omega_{k})$ generate the ring
$\mathcal{R}^{(m)}_{\mathrm{GW}}(\An)$ over the field
$\mathbb{Q}((u))((s_1,\dots,s_n))$.
\end{corstar}

\begin{corstar}
Assuming the generation conjecture, for any three cohomology-weighted partitions, the structure constants
$$\langle \overrightarrow{\mu},\overrightarrow{\nu},\overrightarrow{\rho}\rangle^{\mathrm{Hilb}}_{\An}$$
are rational functions in $q$ and $s_1,\dots, s_n$.
Under the variable substitution $q = -e^{iu}$,
we have
\begin{align*}
\langle\overrightarrow{\mu},\overrightarrow{\nu},\overrightarrow{\rho}\rangle^{\mathrm{Hilb}}_{\An} = 
(-iu)^{-m+l(\mu)+l(\nu)+l(\rho)}\mathsf{Z}^{\prime}(\An\times\mathbf{P}^1)_{\overrightarrow{\mu},\overrightarrow{\nu},\overrightarrow{\rho}}.
\end{align*}
\end{corstar}

This last corollary is the full Gromov-Witten/Hilbert correspondence
for $\An$ surfaces.  Equivalently, under a transcendental change of variables, the Gromov-Witten theory
of $\An\times\mathbf{P}^1$ defines a ring deformation of $H^{\ast}_{T}(\mathrm{Hilb}(\An),\mathbb{Q})$
that is isomorphic to the quantum cohomology ring.

From an algorithmic point of view, 
$\mathrm{Corollary}^{*}$ \ref{gencor}
gives a concrete approach to calculating an arbitrary
three-point invariant of $\An\times\mathbf{P}^1$
in terms of the divisor calculations of proposition \ref{relprop2}.  

Given any Nakajima basis element $\overrightarrow{\rho}$,
let                                                                             
$$M_{\overrightarrow{\rho}}$$                                                                    
denote the matrix of multiplication by $\overrightarrow{\rho}$ in the Nakajima basis for $\mathcal{R}_{\mathrm{GW}}$.  After      
applying the inner product, its entries are precisely the three-point invariants we are trying to compute.  
For $\An$, the              
statement of corollary \ref{gencor} is that the vectors                               
$$M_{(2)}^{a} \cdot \prod M_{(1,\omega_{k})}^{b_k} \cdot (1,....1)$$                                    
span $\mathcal{R}_{\mathrm{GW}}$.  In particular, 
for any $\overrightarrow{\rho}$, we can explicitly             
calculate the linear dependence                                                 
\begin{align*}
\overrightarrow{\rho} = \sum c_{a,b_k} M_{(2)}^{a} \cdot \prod M_{(1,\omega_{k})}^{b_k} \cdot (1,....1),
\end{align*}
with coefficients $c_{a,b_k} \in \mathbb{Q}(t_1,t_2)((u,s_1,\dots,s_n))$.  This implies                                                                    
$$M_{\overrightarrow{\rho}} = \sum c_{a,b_k}  M_{(2)}^{a} \cdot \prod M_{(1,\omega_{k})}^{b_k}.$$     
Finally, we extend the calculation for $k=3$ to arbitrary $k$ with the following proposition.

\begin{propstar}
The $k$-point function 
$$\mathsf{Z}^{\prime}(\An\times\mathbf{P}^1)_{\overrightarrow{\mu_1},\dots,\overrightarrow{\mu_k}}$$
is determined from the case $k=3$.
\end{propstar}
\begin{proof}
If we have relative points $z_1, \dots, z_k$ for $k\geq3$, we consider a degeneration of $\mathbf{P}^1$ to a chain of rational curves of length $r$ with each $z_i$ on a distinct component. The degeneration formula reduces the computation to the individual components, each of which has only three relative points.  If $k=1$ or $2$, we can add relative insertions with weighted partition $(1,\dots,1)$ while leaving the invariant unchanged since this corresponds
to multiplication by the identity.
\end{proof}

\subsection{$\An$-bundles over higher genus curves}
We again assume the generation conjecture in this section.
Given a curve $C$ of genus $g$ equipped with line bundles $L_1, L_2$ of
degrees $a$ and $b$ respectively.  The total space
$$L_1 \oplus L_2$$
admits a fiberwise $T$-action.  In \cite{localcurves},
the Gromov-Witten theory of these noncompact threefolds is calculated 
using the formalism of a $1+1$-dimensional topological quantum field theory.

The above space also admits a fiberwise $\mathbb{Z}_{n+1}$-action which commutes with the $T$-action.  By taking the quotient and passing to the resolution, we obtain the noncompact threefold
$$\mathcal{X}_n(a,b) \longrightarrow(L_1\oplus L_2)/ \mathbb{Z}_{n+1}$$
which is an $\An$-fiber bundle over $C$ which again admits a fiberwise $T$-action.
For $k$ points $z_1,\dots,z_k \in C$ and $k$ cohomology-weighted partitions
$\overrightarrow{mu_1},\dots,\overrightarrow{\mu_k}$, we are interested in the Gromov-Witten theory
of $\mathcal{X}_n(a,b)$ relative to the fibers over $z_1,\dots,z_k$.  This can be encoded in a generating function
\begin{align}\label{relativepartition}
\mathsf{Z}^{\prime}(\mathcal{X}_n(a,b))_{\mu_1,\dots,\mu_k} \in \mathbb{Q}(t_1,t_2)((u))[[s_1,\dots,s_n]].
\end{align}

Given the generation statement, it is again possible to determine the $T$-equivariant Gromov-Witten theory using an enriched TQFT structure.
In \cite{localcurves}, the calculation of (\ref{relativepartition}) for arbitrary $C, a,b$ is reduced the following cases.

\begin{enumerate}
\item $\mathcal{X}_n(0,0)$ relative to $1$,$2$, or $3$ points
These invariants are precisely the relative invariants of $\An\times\mathbf{P}^1$ that we have just calculated, under the assumption of the generation conjecture.

\item $\mathcal{X}_n(0,-1)$ relative to $1$ point
This is less trivial.  For dimension reasons, the only nonzero invariant has cohomological degree $0$
and can thus be computed by any specialzation of the equivariant parameters.  In particular, we can work with the Calabi-Yau
specialization, for which the equivariant parameters sum to $0$ at fixed points away from the relative fiber. 
The computation can then be executed using the topological vertex formalism of \cite{AKMV}, proven in \cite{LLLZ},\cite{MOOP}.
\end{enumerate}

\section{Linear Hodge series}\label{linearhodge}
We apply the rubber evaluation of section \ref{rubberan} to
give a closed evaluation of reduced Gromov-Witten invariants
of $\An$ with a single Hodge class.  Since the extension to 
$\An$ is immediate, we will only write the evaluation in the case of 
$\mathcal{A}_{1}$.
For curve class $d[E]$, we give a formula for the generating function
$$F_{d}(u,z_1,\dots,z_r) =
\sum_{g,a_{1},\dots,a_{r}\geq 0} \langle (-1)^{g}\lambda_{g-\sum a_{i}} \tau_{a_{1}}(\omega)\cdot\dots\cdot\tau_{a_{r}}(\omega)
\rangle_{g,d}^{\mathcal{A}_{1},\mathrm{red}}u^{2g}(-z_{1})^{a_{1}}\cdot\dots\cdot(-z_{r})^{a_{r}}.$$
\begin{thm}\label{hodgethm}
\begin{align*}
F_{d}(u,z_1,\dots,z_r) = \frac{1}{d^{3}\mathcal{S}(du)^{2}}
\prod_{k=1}^{r} \frac{1}{iu}\left[ G(\frac{i\cdot d\cdot z_{k}u}{1-e^{-idu}}, z_{k})
 - G(\frac{-i\cdot d\cdot z_{k}u}{1-e^{idu}}, z_{k})\right]
\end{align*}
where
$$G(w,z) = \sum_{m=1}^{\infty} \frac{w^{m}}{(z)\cdot(z+1)\cdot\dots(z+m)}.$$
\end{thm}

In the above expression, $G(w,z)$ should be expanded in positive powers of $z$. Because of
the factors of $z_{k}$ in our substitution for $w$, the expression gives a well-defined power series in
$z_{1},\dots,z_{k}$.

\subsection{Degree scaling and factorization rule}

The degree dependence from theorem \ref{mainthm} applies here, so we immediately reduce
to the case of $F_{1}(u, z_1,\dots,z_r)$; from now on, we suppress the subscript.

As in the proof of \ref{statprop}, we can degenerate $\mathcal{A}_1$ to a comb of rational curves so that
the spine has normal bundle $\mathcal{O}(-2)$ and each tooth has normal bundle $\mathcal{O}$ and
a single insertion.  The factorization rule 
established there extends to include Hodge classes by restricting the Hodge bundle on $\Mbar_{g}$ 
to its boundary strata.  The resulting factorization rule is
$$F(u, z_1, \dots, z_r) = f_0(u)\prod_{i=1}^r g(u, z_i)$$
where $f_0(u)$ is the contribution of the comb at level $-2$
and $g(u,z_i)$ is the contribution of $\mathbf{P}^1$ relative to $\infty$ at level $0$.  By comparing
with the case of $r=1$, we can remove the dependence on $g(u,z)$:
$$F(u,z_1,\dots, z_r) = f_0(u)^{1-r} \prod_{i=1}^{r} F(u,z_i).$$

As a warm-up calculation, we compute $f_0(u)$.
\begin{prop}
$f_0(u) = \sum (-1)^g\langle \lambda_g\rangle^{\mathrm{red}}_{g} u^{2g} = \frac{1}{\mathcal{S}(u)^2}$
\end{prop} 
\begin{proof}
Consider the connected rubber evaluation from section \ref{rubberan} for the threefold $\mathcal{A}_1\times\mathbf{P}^1$ in degree $(\beta,0)$:
$$\langle\emptyset| \emptyset\rangle^{\sim} = (t_1+t_2)\cdot\frac{1}{\mathcal{S}(u)^{2}}$$
We can rewrite this integral by adding a  $\tau_{0}(\omega)$ insertion using the divisor equation and rigidifying  it by fixing this insertion to lie over a specified point in $\mathbf{P}^1$.  If we then apply localization along the $\mathbf{P}^1$ direction, the answer is precisely the desired Hodge integral.
\end{proof}

\subsection{Auxiliary evaluations}
We now introduce two auxiliary series of Hodge integrals with a single stationary insertion.
Recall $$\Lambda(-1) = (-1)^g \lambda_g + (-1)^{g-1} \lambda_{g-1} \dots +1.$$
The generating functions we evaluate are
$$A_{k}(u) = \sum \langle \Lambda(-1) \frac{\omega}{1-k\psi}\rangle^{\mathrm{red}}_{g}u^{2g}$$
and
$$B_{l}(u) = \sum \langle \Lambda(-1) \prod_{j=1}^{l}(j \psi +1)(\omega)\rangle^{\mathrm{red}}_{g,1}u^{2g}.$$

Each expression is a sum of invariants of the form $\langle\lambda_j \tau_{g-j}(\omega)\rangle$.
We only sum over terms that satisfy the dimension constraint; in particular the first term is a finite sum.  
Finally, we will also need the following series with two stationary insertions.
$$C_{k,l}(u) = \sum \langle \Lambda(-1) \frac{\omega}{1-k\psi_{1}}\prod_{j=1}^{l-1}(j \psi_{2} +1)(\omega)\rangle.$$
Here, $\psi_{1}, \psi_{2}$ denote the cotangent lines at the two marked points.
The factorization rule immediately yields the evaluation of $C$ in terms of $A$ and $B$:
$$C_{k,l}(u) = \frac{A_{k}(u)\cdot B_{l}(u)}{f_{0}(u)}.$$
The nice feature of these generating functions is that they admit simple evaluations via localization arguments.
\begin{prop}
$$A_k(u) = \sum_{j=1}^{k} j! \binom{k-1}{k-j}k^{-j}\mathcal{S}(ju)\mathcal{S}(u)^{j-2}$$
$$B_l(u) = \frac{\mathcal{S}((l+1)u)}{\mathcal{S}(u)^{l+1}}.$$ 
\end{prop}
\begin{proof}
Consider the threefold $A_1\times\mathbf{P}^1$ relative to $A_1\times {\infty}$, equipped with the $\mathbb{C}^{\ast}$-action 
from the $\mathbf{P}^1$. 
We will derive the two identities by applying relative localization with respect to this torus action.  Throughout this argument, 
we use the analysis of possible localization configurations that was required in the proof of Proposition \ref{rubberpropan}.
Let
$$[F_0], [F_{\infty}]$$
denote the equivariant classes of the fibers over the fixed points of $\mathbf{P}^1$
with tangent weights $1$ and $-1$.  We consider relative stable maps with target homology 
class $(\beta, m)$ for $m > 0$.
 
 Since $[F_{\infty}]^{2} = 0$, we have the vanishing statement 
 $$\langle \tau_{0}(\omega\cdot[F_{\infty}]^2)\prod_{i=1}^{m} \tau_{0}(\omega[F_0])| \mu(\omega)\rangle^{\mathcal{A}_{1}\times\mathbf{P}^1}_{(\beta,m)} =0$$
for $\mu = (m)$.
When we apply relative localization with respect to the torus action, the fixed loci have the following structure.   

As always, there is a unique irreducible component which maps nontrivially to 
$\mathcal{A}_1$.
If the primary component maps to $A_1 \times \infty$, then the fixed locus consists of a degenerate target which contributes a rubber integral.  Each of the $m$ distinct points mapping to $F_0$ must lie
on a distinct rational tail because they are fixed with $\omega$-insertions.  The only possible
contribution is 
$$(m!)\cdot \langle 1^m(\omega)|\frac{1}{1-\Psi}\cdot \tau_{0}(1)| (m,\omega)\rangle.$$
In the above formula, $\Psi$ represents the cotangent line to the moduli space of degenerations of the nonrigid target.  The insertion $\tau_0(1)$ arises from the marked point with the $F_{\infty}$ insertions.  The factorial contribution occurs because the partition has ordered parts.  After applying the string equation and our rubber evaluation, the contribution of this term to the localization sum is
$$m\mathcal{S}(mu)\mathcal{S}(u)^{m-2}.$$

If the primary component maps to $\mathcal{A}_1\times 0$, then the allowed fixed loci are described as follows.
There is a single rational fiber tail of degree $a < m$  attached to the primary component.  In the degenerate part of the target, there is a genus $0$ curve with ramification profile $(m)$ over $\infty$
and profile $(a,\rho)$ over $0$ for some partition $\rho$ of $m-a$.  Finally, there is a rational curve
for each part of $\rho$.  Since the relative insertion has an $\omega$-insertion, all
the other marked points must lie on the primary component.

The contribution of this term is 
$$ 
-A_{a}(u)\cdot \frac{a^{a+1}}{a!} \frac{1}{\mathrm{Aut}\rho} \prod \frac{\rho_{i}^{\rho_{i}-1}}{\rho_{i}!}
(-m)^{l(\rho)}
$$
using genus 0 Hurwitz evaluations for the rational tail contributions.  The $m$ marked points
on the primary component can be removed with the divisor equation.
The summation over $\rho$ is handled by the identity
$$\sum_{\rho} \frac{1}{\mathrm{Aut}\rho} \prod \frac{\rho_{i}^{\rho_{i}-1}}{\rho_{i}!}
(-m)^{l(\rho)} = \frac{-m(-a)^{m-a-1}}{m-a)!}.$$

We thus have the identity
$$\mathcal{S}(mu)\mathcal{S}(u)^{m-2} = 
\frac{1}{m!}\sum_{a} (-1)^{m-a}A_{a}(u)\binom{m}{a}a^{m}$$
which is easily inverted to yield the first statement.

For the second part of the proposition, we study the relative invariant in degree $(\beta,m)$
$$\langle \prod_{j=1}^{m-1}(j \psi_{1} +1)(\omega \cdot  [F_0]) \tau_0(\omega [F_{\infty}]^2) | (m,\omega)\rangle^{\mathcal{A}_{1}\times\mathbf{P}^1}_{(\beta,m)}= 0$$
which again vanishes for trivial reasons.  Our analysis proceeds as before.  If the primary
component maps to the degenerate part, then the first insertion forces a unique rational tail
of degree $m$.  Indeed, a tail of smaller degree would give a vanishing contribution in the localization
expression for the first insertion.  The contribution is now
$$m\mathcal{S}(mu)^{2}\mathcal{S}(u)^{-2}.$$
In the other fixed loci, the primary component maps to $A_1 \times 0$ with a 
configuration of rational tails identical to the last computation.  The only difference is that we have a more complicated insertion on the primary component .  The contribution is
$$\sum_{a \leq m,\rho} C_{a,m-1}(u)\cdot \frac{a^{a+1}}{a!} \frac{1}{\mathrm{Aut}\rho} \prod \frac{\rho_{i}^{\rho_{i}-1}}{\rho_{i}!}
(-m)^{l(\rho)}.$$
By applying the factorization rule and our evaluation for $A_{a}(u)$, this is precisely
$$\frac{m\cdot B_{m-1}(u)}{f_{0}(u)}\mathcal{S}(mu)\mathcal{S}(u)^{m-2}.$$
Since the two fixed loci sum to zero, this gives the identity for $B_{m-1}(u)$.
\end{proof}

\subsection{Proof of Theorem $\ref{hodgethm}$}
\begin{proof}
In order to evaluate $F(u,z)$, we expand $B_{m}(u)$ in monomial form and invert
the resulting system.  More precisely, 
if $$F(u,z) = \sum_{m} F_{m}(u) z^{m}$$
then
$$B_{m}(u) = m!\cdot\sum_{k=0}^{m}e_{m-k}(1, \frac{1}{2},\dots, \frac{1}{m})F_{k}(u).$$
Here, $e_{j}(a_{1},\dots,a_{r})$ is the $j$-th elementary symmetric function.
This inverts to give 
$$F_{m}(u) = \sum_{k=0}^{m}(-1)^{m-k} \frac{B_{k}(u)}{k!}h_{m-k}(1,\frac{1}{2},\dots,\frac{1}{k})$$
where $h_{j}(a_{1},\dots,a_{r})$ is the $j$-th complete symmetric function.
If we sum over $m$ and use the evaluation for $B_{k}(u)$, we have
\begin{align*}
F(u,z) = \frac{1}{z}\sum_{k\geq 1} z^{k}\frac{S(ku)}{(k-1)!S(u)^{k}}
\prod_{i=1}^{k}\frac{1}{1+z/i}. 
\end{align*}
This is equivalent to the expression in the theorem statement.
\end{proof}

\vspace{+10 pt}
Department of Mathematics\\
Columbia University\\
New York, NY 10027, USA\\
dmaulik@cpw.math.columbia.edu

\end{document}